\documentclass[11pt,a4paper]{article}

\usepackage{graphicx}
\usepackage{amsmath}
\usepackage{amsfonts}
\usepackage{amssymb}
\usepackage{enumerate}
\usepackage{lscape}
\usepackage{longtable}
\usepackage{rotating}
\usepackage{color}
\usepackage{url}
\usepackage{subfigure}
\usepackage{rotating}
\usepackage[title,titletoc,toc]{appendix}
\usepackage{mathtools}
\newtheorem{theorem}{Theorem}[section]

\usepackage[ruled,vlined,noline,linesnumbered]{algorithm2e}
\usepackage{algorithmic}

\newtheorem{lemma}{Lemma}[section]

\newtheorem{proposition}{Proposition}[section]

\newtheorem{assumption}{Assumption}[section]

\newenvironment{proof}[1][Proof]{\textbf{#1.} }{\ \rule{0.5em}{0.5em} \vspace{1ex}}

\def\real{\mathbb{R}}

\DeclareMathOperator{\sign}{sgn}

\DeclareMathOperator{\cc}{c}

\DeclareMathOperator{\NT}{N}

\newcommand{\diag}{\operatorname{diag}}

\begin{document}

\title{A Line-Search Algorithm Inspired by the Adaptive Cubic Regularization Framework and Complexity Analysis}
\author{
E. Bergou \thanks{MaIAGE, INRA, Universit\'e Paris-Saclay, 78350 Jouy-en-Josas, France
 ({\tt elhoucine.bergou@inra.fr}).
}
\and
Y. Diouane\thanks{Institut Sup\'erieur de l'A\'eronautique et de l'Espace (ISAE-SUPAERO), Universit\'e de Toulouse, 31055 Toulouse Cedex 4, France
 ({\tt youssef.diouane@isae.fr}).
}
\and
S. Gratton\thanks{INP-ENSEEIHT, Universit\'e de Toulouse,
31071 Toulouse Cedex 7, France ({\tt serge.gratton@enseeiht.fr}).}
}
\maketitle
\footnotesep=0.4cm
\begin{abstract}

Adaptive regularized framework using cubics has emerged as an alternative to line-search and trust-region algorithms for smooth nonconvex optimization, with an optimal complexity amongst second-order methods. In this paper, we propose and analyze the use of an iteration dependent scaled norm in the adaptive regularized framework using cubics. Within such scaled norm, the obtained method behaves as a line-search algorithm along the quasi-Newton direction with a special backtracking strategy. Under appropriate assumptions, the new algorithm enjoys  the same convergence and complexity properties as adaptive regularized algorithm using cubics. The complexity for finding an approximate first-order stationary point can be improved to be optimal whenever a second order version of the proposed  algorithm is regarded. In a similar way, using the same scaled norm to define the trust-region neighborhood, we show that the trust-region algorithm behaves as a line-search algorithm. The good potential of the obtained algorithms is shown on a set of large scale optimization problems.

\end{abstract}

\bigskip

\begin{center}
\end{center}
\section{Introduction}
An unconstrained nonlinear optimization problem considers the minimization of a scalar function known as the objective function. Classical iterative methods for solving the previous problem are trust-region (TR) \cite{Conn_Gould_Toin_2000,YYuan_2015}, line-search (LS)  \cite{JEDennis_RBSchnabel_1983} and algorithms using cubic regularization. The latter class of algorithms has been first investigated by Griewank \cite{AGriewank_1981} and then by Nesterov and Polyak \cite{YNesterov_BTPolyak_2006}. Recently, Cartis \textit{et al} \cite{CCartis_NIMGould_PhLToint_2011_a} proposed a generalization to an adaptive regularized framework using cubics (ARC). 

The worst-case evaluation complexity of finding an $\epsilon$-approximate first-order critical point using TR or LS methods is shown to be computed in at most 
  $\mathcal{O}(\epsilon^{-2})$ objective function or gradient evaluations, where $\epsilon \in ]0,1[$ is a user-defined accuracy threshold on the gradient norm \cite{YNesterov_2004,SGratton_ASartenaer_PhLToint_2008,CCartis_PhLSampaio_PhLToint_2015}. Under appropriate assumptions, ARC takes at most $\mathcal{O}(\epsilon^{-3/2})$ objective function or gradient evaluations to reduce the gradient of the objective function norm below $\epsilon$, and thus it is improving substantially the 
worst-case  complexity over the classical TR/LS methods~\cite{CCartis_NIMGould_PhLToint_2011_b}. Such complexity bound can be improved using higher order regularized models, we refer the reader for instance to the references \cite{Birgin2016,CCartis_NIMGould_PhLToint_2017}.

More recently, a non-standard TR method~\cite{CurtRobiSama16} is  proposed with the same worst-case complexity bound as ARC. 
It is proved also that the same worst-case complexity $\mathcal{O}(\epsilon^{-3/2})$  can be achieved by mean of a specific variable-norm in a TR method \cite{Martinez2017} or using quadratic regularization \cite{Birgin2017}. All previous approaches use a cubic sufficient descent condition instead of the more usual predicted-reduction based descent. Generally, they need to solve more than one linear system in sequence at each outer iteration (by outer iteration, we mean the sequence of the iterates generated by the algorithm), this makes the computational cost per iteration expensive. 

In \cite{Bergou_Diouane_Gratton_2017}, it has been shown how to use the so-called energy norm in the ARC/TR framework when a symmetric positive definite (SPD) approximation of the objective function Hessian is available. 
Within the energy norm,  ARC/TR methods behave as LS algorithms along the Newton direction, with a special backtracking strategy and an acceptability condition in the spirit of ARC/TR methods. As far as the model of the objective function is convex, in \cite{Bergou_Diouane_Gratton_2017}  the proposed LS algorithm derived from ARC   enjoys the same convergence and complexity analysis properties as ARC, in particular the first-order complexity bound of $\mathcal{O}(\epsilon^{-3/2})$. In the complexity analysis  of ARC method~\cite{CCartis_NIMGould_PhLToint_2011_b}, it is required that the Hessian  approximation has to approximate accurately enough the true Hessian \cite[Assumption AM.4]{CCartis_NIMGould_PhLToint_2011_b}, obtaining such convex approximation may be out of reach when handling nonconvex optimization. This paper generalizes the proposed methodology in \cite{Bergou_Diouane_Gratton_2017}  to handle nonconvex models. We propose to use, in the regularization term of the ARC cubic model, an iteration dependent scaled norm. In this case, ARC behaves as an LS algorithm with a worst-case evaluation complexity of finding an $\epsilon$-approximate first-order critical point of $\mathcal{O}(\epsilon^{-2})$ function or gradient evaluations. Moreover, under appropriate assumptions, a second order version of the obtained LS algorithm is shown to have a worst-case complexity of $\mathcal{O}(\epsilon^{-3/2})$.

The use of a scaled norm was first introduced in \cite[Section 7.7.1]{Conn_Gould_Toin_2000} for TR methods where it was suggested to use the absolute-value
of the Hessian matrix in the scaled norm, such choice was described as ``the ideal trust region'' that reflects the proper scaling of the underlying problem. For a large scale indefinite Hessian matrix, computing its absolute-value is certainly a computationally expensive task as it requires a spectral decomposition. This means that for large scale optimization problems the use of the absolute-value based norm can be seen as out of reach. Our approach in this paper is different as it allows the use of subspace methods.

In fact, as far as the quasi-Newton direction is not orthogonal with the gradient of the objective function at the current iterate, the specific choice of the scaled norm renders the ARC subproblem solution collinear with the quasi-Newton direction. Using subspace methods, we also consider the large-scale setting when the matrix factorizations are not affordable, implying that only iterative methods for computing a trial step can be used.  
Compared to the classical ARC, when using the Euclidean norm, the dominant computational cost regardless the function evaluation cost of the resulting algorithm is mainly the cost of solving a linear system for successful iterations. Moreover, the cost of the subproblem solution for unsuccessful iterations is getting inexpensive and requires only an update of a scalar. Hence, ARC behaves as an LS algorithm along the quasi-Newton direction, with a special backtracking strategy and an acceptance criteria in the sprite of ARC algorithm. 

In this context, the obtained LS algorithm is globally convergent and  requires a number of iterations of order $\epsilon^{-2}$ to produce an $\epsilon$-approximate first-order critical point. A second order version of the algorithm is also proposed, by making use of the exact Hessian or at least of a good approximation of the exact Hessian, to ensure an optimal worst-case complexity bound of order $\epsilon^{-3/2}$. In this case, we investigate how the complexity bound depends on the quality of the chosen quasi-Newton direction in terms of being a sufficient descent direction. In fact, the obtained complexity bound can be worse than it seems to be whenever the quasi-Newton direction is approximately orthogonal with the gradient of the objective function. Similarly to ARC, we show that the TR method behaves also as an LS algorithm using the same scaled norm as in ARC. Numerical illustrations over a test set of large scale optimization problems are given in order to assess the efficiency of the obtained LS algorithms.

The proposed analysis in this paper assumes that the quasi-Newton direction is not orthogonal with the gradient of the objective function during the minimization process. 
When such assumption is violated, one can either modify the Hessian approximation using regularization techniques or, when a second order version of the LS algorithm is regarded, switch to the classical ARC algorithm using the Euclidean norm until this assumption holds. In the latter scenario, we propose to check first if there exists an approximate quasi-Newton direction, among all the iterates generated using a subspace method, which is not orthogonal with the gradient and that satisfies the desired properties. If not, one minimizes the model using the Euclidean norm until a new successful outer iteration is found. 

We organize this paper as follows. In Section~\ref{section:1}, we introduce the ARC method using a general scaled norm and derive the obtained LS algorithm on the base of ARC when a specific scaled norm is used. 
Section~\ref{section:2} analyses the minimization of the cubic model and discusses the choice of the scaled norm that simplifies solving the ARC subproblem. Section~\ref{section:3} discusses first how the iteration dependent  can be chosen uniformly equivalent to the Euclidean norm, and then we propose a second order LS algorithm that enjoys the optimal complexity bound. The section ends with a detailed complexity analysis of the obtained algorithm. Similarly to ARC and using the same scaled norm, an LS algorithm in the spirit of TR algorithm is proposed in Section~\ref{section:4}. Numerical tests are illustrated and discussed in Section~\ref{section:5}. Conclusions and future improvements are given in Section~\ref{section:6}.

\section{ARC Framework Using a Specific $M_k$-Norm }
\label{section:1}
\subsection{ARC Framework}
 We consider a problem of unconstrained minimization of the form
\begin{eqnarray}
\label{nl_ls_problem}
\displaystyle \min_{x \in \real^n} f(x), 
\end{eqnarray}
 where the objective function $f: \real^n \rightarrow \mathbb{R}$ is assumed to be continuously differentiable. The ARC framework \cite{CCartis_NIMGould_PhLToint_2011_a} can be described as follows: at a given iterate $x_k$, we define $m^Q_k : \real^n \rightarrow \real$  as an approximate second-order Taylor approximation of the objective function $f$ around $x_k$, i.e.,
 \begin{eqnarray}
\label{q_model}
\displaystyle m^Q_k(s) & =&  f(x_k) + s^{\top}g_k + \frac{1}{2} s^{\top} B_k s,
\end{eqnarray}
where $g_k = \nabla f(x_k)$ is the gradient of $f$ at the current iterate $x_k$, and $B_k$ is a symmetric  local approximation (uniformly bounded from above) of the Hessian of $f$ at $x_k$. 
The trial step $s_k$ approximates the  global minimizer of the cubic model $m_k(s)= m^Q_k(s) + \frac{1}{3}\sigma_k \| s \|^3_{M_k}$, i.e.,
 \begin{eqnarray}
\label{eq:nl_ARC_subproblem}
s_k \approx \displaystyle \arg \min_{s \in \real^n}  m_k(s),
\end{eqnarray}
where $\|.\|_{M_k}$ denotes an iteration dependent scaled norm  of the form $\|x\|_{M_k}= \sqrt{x^{\top}M_k x}$ for all $x \in \real^n$ and $M_k$ is a given SPD matrix. $\sigma_k >0$ is a dynamic positive parameter that might be regarded as the reciprocal
of the TR radius in TR algorithms (see \cite{CCartis_NIMGould_PhLToint_2011_a}). 
The parameter $\sigma_k$ is  taking into account the agreement between the objective function $f$ and the model $m_k$.  

To decide whether the trial step is acceptable or not a ratio between the actual reduction and the predicted reduction is computed, as follows:
\begin{equation} 
\label{rho}
\rho_k \; = \;  \frac{f(x_k) - f(x_k + s_k)}{ f(x_k) - m^Q_k(s_k)}.
\end{equation}
For a given scalar $0 < \eta < 1$, the $k^{th}$ outer iteration will be said \textit{successful} if $\rho_k \ge \eta$, and  \textit{unsuccessful} otherwise.  For all \textit{successful} iterations we set $x_{k+1}= x_k + s_k$; otherwise the current  iterate is kept unchanged $x_{k+1} = x_k$. 
We note that, unlike the original ARC \cite{CCartis_NIMGould_PhLToint_2011_a,CCartis_NIMGould_PhLToint_2011_b} where the cubic model is used to evaluate  the denominator in (\ref{rho}), in the nowadays works related to ARC, only the quadratic approximation $m^Q_k(s_k)$ is used in the comparison with the actual value of $f$ without the regularization parameter (see \cite{Birgin2016} for instance).
Algorithm~\ref{algo:ARC} gives a detailed description of ARC.

 \LinesNumberedHidden
\begin{algorithm}[!ht]
\SetAlgoNlRelativeSize{0}
\caption{\bf ARC algorithm.}
\label{algo:ARC}
\SetAlgoLined
\KwData{select an initial point $x_0$ and the constant $0<  \eta<1$. Set the initial regularization $\sigma_0 >0$ and $\sigma_{\min} \in ]0, \sigma_0]$, set also 
the constants  $0< \nu_1 \le 1 <\nu_2 $.}
\For{$k= 1,  2, \ldots$}{
Compute the step $s_k$ as an approximate solution of (\ref{eq:nl_ARC_subproblem}) such that 
 \begin{eqnarray}
\label{Cauchy_decrease}
m_k(s_k) &\le & m_k(s^{\cc}_{k})
\end{eqnarray}
where $s^{\cc}_{k} = - \delta^{\cc}_{k} g_k$ and $\delta^{\cc}_{k} = \displaystyle \arg \min_{t > 0} m_k(-t g_k)$
\;
\eIf{$\rho_k \ge \eta$}{Set 
$x_{k+1}= x_k + s_k$ and $ \sigma_{k+1} = \max\{\nu_1 \sigma_k,\sigma_{\min}\}$\; }{Set $x_{k+1} = x_k$ and $\sigma_{k+1}= \nu_2\sigma_k$\;}
}
\end{algorithm}

The Cauchy step $s^{\cc}_{k}$, defined in Algorithm \ref{algo:ARC}, is computationally inexpensive compared to the computational cost of the global minimizer of $m_k$. The condition (\ref{Cauchy_decrease}) on $s_k$ is sufficient for ensuring global convergence of ARC to first-order critical points. 

From now on, we will assume that first-order stationarity is not reached yet, meaning that the gradient of the objective function is non null at the current iteration $k$ (i.e., $g_k \neq  0$). Also, $\| \cdot \|$ will denote the vector or matrix $\ell_2$-norm, $\sign(\alpha)$ the sign of a real $\alpha$, and $I_n$ the identity matrix of size $n$.

\subsection{An LS Algorithm Inspired by the ARC Framework}

Using a specific $M_k$-norm in the definition of the cubic model $m_k$, we will show that ARC framework (Algorithm \ref{algo:ARC}) behaves as an LS algorithm along the quasi-Newton direction. 
In a previous work \cite{Bergou_Diouane_Gratton_2017}, when the matrix $B_k$ is assumed to be positive definite, we showed that the minimizer  $s_k$ of the cubic model defined in (\ref{eq:nl_ARC_subproblem}) is getting collinear with the quasi-Newton direction when the matrix $M_k$ is set to be equal to $B_k$.
In this section we generalize our proposed approach  to cover the case where the linear system $B_ks = -g_k$  admits an approximate solution and $B_k$ is not necessarily SPD.  

Let $s^Q_k$ be an approximate solution of the linear system $B_k s=-g_k$ and assume that such step $s_k^Q$ is not orthogonal with the gradient of the objective function at $x_k$, i.e., there exists an $\epsilon_d>0$ such that $|g_k^{\top}s_k^Q | \ge \epsilon_d \|g_k \| \| s_k^Q\| $. Suppose that there exists an SPD matrix $M_k$ such that $M_ks^Q_k= \frac{\beta_k \|s^Q_k\|^2}{g_k^{\top}s_k^Q}g_k$ where $\beta_k \in ]\beta_{\min}, \beta_{\max}[$ and $\beta_{\max}> \beta_{\min}>0$, in Theorem \ref{th:equivnm}  we will show that such matrix $M_k$ exists. By using the associated $M_k$-norm in the definition of the cubic model $m_k$, one can show (see Theorem  \ref{cor:1}) that an approximate stationary  point of the subproblem (\ref{eq:nl_ARC_subproblem})  is of the form 
\begin{eqnarray}
 \label{eqdeltaarc}
s_{k} & = & \delta_k s^Q_k, ~~\mbox{where}~~  \delta_k=  \frac{2 }{1-\sign(g_k^{\top}s^Q_k) \sqrt{1 +4 \frac{\sigma_k \| s^Q_k \|^3_{M_k}}{|g_k^{\top}s^Q_k|} }} .
\end{eqnarray}

For \textit{unsuccessful} iterations in Algorithm~\ref{algo:ARC}, since the step direction $s^Q_k$ does not change, the approximate solution of the subproblem,  given by (\ref{eqdeltaarc}), can be obtained only by updating the step-size $\delta_k$.  This means that the subproblem computational cost of \textit{unsuccessful} iterations is getting straightforward compared to solving the subproblem as required by ARC when the Euclidean norm is used (see e.g., \cite{CCartis_NIMGould_PhLToint_2011_a}).  As a consequence, the use of the proposed  $M_k$-norm in Algorithm~\ref{algo:ARC} will lead to a new formulation of ARC algorithm where the dominant computational cost, regardless the objective function evaluation cost, is  the cost of solving a linear system for \textit{successful} iterations.  In other words, with the proposed $M_k$-norm, the ARC algorithm behaves as an LS method with a specific backtracking strategy and an acceptance criteria in the sprite of ARC algorithm, i.e., the step is of the form $s_{k}= \delta_k s_k^Q  $  where  the step length $\delta_k>0$ is chosen such as
\begin{eqnarray} \label{sdcond:arc}
\rho_k=\frac{f(x_k) - f(x_k +s_{k} )}{ f(x_k) - m_k^Q(s_{k})} \ge  \eta & \mbox{and}& m_k(s_{k}) \le m_k(-\delta^{\cc}_k g_k).
\end{eqnarray}
The step lengths $\delta_k$ and $\delta^{\cc}_k$ are computed respectively as follows:
\begin{eqnarray} 
\delta_k&= &\frac{2}{1-\sign(g_k^{\top}s_k^Q)\sqrt{1 + 4 \frac{\sigma_k \beta_k^{3/2} \| s^Q_k \|^3}{|g_k^{\top}s_k^Q|}}} \label{eq:deltasigma:1},
\end{eqnarray} 
and
\begin{eqnarray} 
 \delta^{\cc}_k&= &\frac{2}{\frac{g_k^{\top}B_kg_k}{\|g_k \|^2} +\sqrt{\left(\frac{g_k^{\top}B_kg_k}{\|g_k \|^2}\right)^2 + 4 \sigma_k  \chi_k^{3/2} \| g_k \| }} \label{eq:deltasigma:2}
\end{eqnarray} 
where $\chi_k= \beta_k \left(\frac{5}{2}  - \frac{3}{2} \cos(\varpi_k)^2 + 2 \left(\frac{1- \cos(\varpi_k)^2}{\cos(\varpi_k)}\right)^2\right)$ and $\cos(\varpi_k) = \frac{g_k^{\top}s_k^Q}{\|g_k\| \| s_k^Q\|}$. 
The $M_k$-norms of the vectors $s^Q_k$ and $g_k$ in the computation of $ \delta_k$ and $ \delta^{\cc}_k$ have been substituted using the expressions given in Theorem \ref{th:equivnm}. 
The value of $\sigma_k$ is set equal to the current value of the regularization parameter as in the original
ARC algorithm. For large values of $\delta_k$ the  decrease condition (\ref{sdcond:arc}) may not be satisfied. In this case, the value of $\sigma_k$ is enlarged using an expansion factor $\nu_2>1$. Iteratively, the value of $\delta_k$ is updated and the acceptance condition (\ref{sdcond:arc}) is checked again, until its satisfaction.

We referred the ARC algorithm, when the proposed scaled $M_k$-norm is used, by LS-ARC as it behaves as an LS type method in this case.
 Algorithm~\ref{algo:LS-ARC} details the final algorithm.  We recall again that this algorithm is nothing but ARC algorithm using a specific $M_k$-norm. 
 \LinesNumberedHidden
\begin{algorithm}[!ht]
\SetAlgoNlRelativeSize{0}
\caption{\bf LS-ARC algorithm.}
\label{algo:LS-ARC}
\SetAlgoLined
\KwData{select an initial point $x_0$ and the constants $0<  \eta<1$, $0<\epsilon_d<1$, $0< \nu_1 \le 1 <\nu_2 $ and $0< \beta_{\min}< \beta_{\max}$. Set $\sigma_0 >0$ and $\sigma_{\min} \in ]0, \sigma_0]$.}
\For{$k= 1,  2, \ldots$}{
Choose a parameter $\beta_k \in ]\beta_{\min}, \beta_{\max}[$\; 
Let $s^Q_k$ be an approximate solution of $B_k s=-g_k$ such as $|g_k^\top s^Q_k| \ge  \epsilon_d  \|g_k\|\|s^Q_k\|$ \;
Set  $\delta_k$ and $\delta^{\cc}_k$  respectively using (\ref{eq:deltasigma:1}) and (\ref{eq:deltasigma:2})\;
\While{ {\normalfont condition \text{(\ref{sdcond:arc})} is not satisfied}}{
Set $\sigma_k \leftarrow  \nu_2 \sigma_k $ and update $\delta_k$ and $\delta^{\cc}_k$  respectively using (\ref{eq:deltasigma:1}) and (\ref{eq:deltasigma:2})\;
    }
    
    Set $s_k= \delta_k s^Q_k$, $x_{k+1}= x_k + s_k$ and $ \sigma_{k+1} = \max \{\nu_1 \sigma_k,\sigma_{\min}\}$\;
}
\end{algorithm}

Note that in Algorithm \ref{algo:LS-ARC} the step $s^Q_k$ may not exist or be approximately orthogonal with the gradient $g_k$. 
A possible way to overcome this issue can be ensured by modifying the matrix $B_k$ using regularization techniques. In fact, as far as the Hessian approximation is still uniformly bounded from above, the global convergence will still hold as well as a complexity bound of order $\epsilon^{-2}$ to drive the norm of the gradient below $\epsilon \in ]0,1[$ (see \cite{CCartis_NIMGould_PhLToint_2011_a} for instance). 

The complexity bound can be improved to be of the order of $\epsilon^{-3/2}$ if a second order version of the algorithm LS-ARC is used, by making $B_k$ equals to the exact Hessian or at least being a good approximation of the exact Hessian (as in Assumption \ref{asm:B} of Section \ref{section:3}). In this case, modify the matrix $B_k$ using regularization techniques such that the step $s_k^Q$ approximates the linear system $B_ks = -g_k$ and $|g_k^\top s^Q_k| \ge  \epsilon_d  \|g_k\|\|s^Q_k\|$ is not trivial anymore. This second order version of the algorithm LS-ARC will be discussed in details in Section \ref{section:3} where convergence and complexity analysis, when the proposed $M_k$-norm is used, will be outlined.

\section{On the Cubic Model Minimization}
\label{section:2}
In this section, we assume that the linear system $B_ks = -g_k$ has a solution. 
We will mostly focus on the solution of the subproblem (\ref{eq:nl_ARC_subproblem}) for a given outer iteration $k$.  In particular, we will explicit the condition to impose on the matrix $M_k$ in order to get the solution of the  ARC subproblem collinear with the  step $s^Q_k$.  Hence, in such case, one can get the solution of the  ARC subproblem at a modest computational cost. 

The step $s^Q_k$ can be obtained exactly using a direct method if the matrix $B_k$ is not too large. Typically, one can use the $LDL^T$ factorization to solve this linear system. For large scale optimization problems, computing $s^Q_k$ can be prohibitively computationally expensive. We will show that it will be possible to relax this requirement by letting the step $s_k^Q$ be only an approximation of the exact solution using subspace methods.  

In fact, when an approximate solution is used and as far as the global convergence of Algorithm~\ref{algo:ARC} is concerned, all what is needed is that the solution of the subproblem (\ref{eq:nl_ARC_subproblem}) yields a decrease in the cubic model which is as good as the Cauchy decrease (as emphasized  in condition  (\ref{Cauchy_decrease})). 
In practice, a version of Algorithm~\ref{algo:LS-ARC} solely based on the  Cauchy
step would suffer from the same drawbacks as the steepest descent algorithm on ill-conditioned problems and faster convergence can be expected if the matrix $B_k$ influences also the minimization direction. The main idea consists of achieving a further decrease on the cubic model, better than the Cauchy decrease, by projection onto a sequence of embedded Krylov subspaces. We now show how to use a similar idea to compute a solution of the subproblem that is  computationally cheap and  yields the global convergence of Algorithm~\ref{algo:LS-ARC}. 

A  classical way to approximate the exact solution $s^Q_k$ is by using subspace methods, typically a Krylov subspace method. For that, let $\mathcal{L}_k$ be a subspace of $\real^n$ and $l$ its dimension. Let $Q_k$ denotes an $n \times l$ matrix whose columns form a basis of $\mathcal{L}_k$. Thus for all $s \in \mathcal{L}_k$, we have $s_k=Q_kz_k$, for some $z_k \in \real^l$. In this case, $s^Q_k$ denotes the exact stationary point of the model function $m^Q$ over the subspace $\mathcal{L}_k$ when it exists.

For both cases, exact and inexact, we will assume that the step $s_k^Q$ is not orthogonal with the gradient $g_k$.
In what comes next, we state our assumption on $s_k^Q$ formally as follows: 
\begin{assumption}
\label{asm:1}
The model $m_k^Q$ admits a stationary point $s_k^Q$ such as $|g_k^{\top}s_k^Q | \ge \epsilon_d \|g_k \| \| s_k^Q\| $ where $\epsilon_d>0$ is a pre-defined positive constant. 
\end{assumption}
We define also  a Newton-like step $s_k^{\NT}$ associated with the minimization of the cubic model $m_k$ over the subspace $\mathcal{L}_k$ on the following way, when $s^Q_k$ corresponds to the exact solution of $B_ks=-g_k$ by 
  \begin{eqnarray}
\label{Newton_step_ARC}
s_k^{\NT} =  \delta_k^{\NT} s^Q_k,  ~~\mbox{where} ~~ \delta^{\NT}_k=  \arg \min_{\delta \in \mathcal{I}_k } m_k(\delta s_k^{Q}),
\end{eqnarray}
where $\mathcal{I}_k=\real_+$ if $ g_k^{\top}  s_k^{Q} < 0$ and $\mathcal{I}_k=\real_-$ otherwise.
If $s_k^Q$ is computed using an iterative subspace method, then
$
s^{\NT}_k = Q_k z^{\NT}_k,
$
where $z^{\NT}_k$ is the Newton-like step , as in (\ref{Newton_step_ARC}), associated to the following reduced subproblem:
   \begin{eqnarray}
 \label{reduce_model_arc}
 \min_{z\in \real^l} f(x_k) + z^{\top} Q_k^{\top} g_k + \frac{1}{2} z^{\top} Q_k^{\top} B_k Q_k z + \frac{1}{3}\sigma_k \| z \|^3_{Q_k^{\top} M_k Q_k}.
  \end{eqnarray}
\begin{theorem} \label{th:1} Let Assumption \ref{asm:1} hold. The Newton-like step $s^{\NT}_k$ is of the form
 \begin{eqnarray}
 \label{eqdeltaarcN}
 s_k^{\NT} =  \delta_k^{\NT} s^Q_k,  ~~\mbox{where}~~ \delta_k^{\NT}& = & \frac{2 }{1-\sign(g_k^{\top}s_k^Q) \sqrt{1 +4 \frac{\sigma_k \| s_k^Q \|^3_{M_k}}{|g_k^{\top}s^Q_k|} }} .
\end{eqnarray}
\end{theorem}
\begin{proof}
Consider first the case where the step $s^Q_k$ is computed exactly (i.e., $B_ks_k^Q=-g_k$). In this case, for all $\delta \in \mathcal{I}_k$, one has
 \begin{eqnarray}
 \label{min_N_ARC_init}
m_k(\delta s_k^Q) - m_k(0) &=&  \delta g_k^{\top}s^{Q}_k + \frac{\delta^2}{2} [s^{Q}_k]^{\top} B_k [s^{Q}_k] + \frac{\sigma |\delta|^3 }{3} \| s^{Q}_k \|^3_{M_k} \nonumber \\ 
&=&  (g_k^{\top}s^{Q}_k)   \delta  -  (g_k^{\top}s_k^Q) \frac{\delta^2}{2} + (\sigma_k \| s_k^Q \|^3_{M_k}) \frac{ |\delta|^3}{3}.
\end{eqnarray}

If $g_k^{\top}s_k^Q<0$ (hence, $\mathcal{I}_k=\real_+$), we compute the value of the parameter $\delta^N_k$ at which the unique minimizer of the above function is attained. 
Taking the derivative of (\ref{min_N_ARC_init})  with respect to $\delta$ and equating the result to zero, one gets
 \begin{eqnarray}
 \label{equ:n:1}
0 & =& g_k^{\top}s^{Q}_k  -  (g_k^{\top}s_k^Q) \delta^{\NT}_k + \sigma_k \| s_k^Q \|^3_{M_k} \left(\delta^{\NT}_k\right)^2 ,
\end{eqnarray}
and thus, since $ \delta^{\NT}_k>0$,
 \begin{eqnarray*}
\delta^{\NT}_k& =& \frac{g_k^{\top}s_k^Q+ \sqrt{ \left(g_k^{\top}s_k^Q\right)^2 - 4 \sigma_k (g_k^{\top}s_k^Q)\| s_k^Q \|^3_{M_k} }}{ 2 \sigma_k \| s_k^Q \|^3_{M_k}} = \frac{2 }{1 +  \sqrt{1 - 4 \frac{\sigma_k \| s_k^Q \|^3_{M_k}}{g_k^{\top}s^Q_k} }}.
\end{eqnarray*}

If $g_k^{\top}s_k^Q>0$  (hence, $\mathcal{I}_k=\real_-$),and again by taking the derivative of (\ref{min_N_ARC_init})  with respect to $\delta$ and equating the result to zero, one gets
 \begin{eqnarray}
 \label{equ:n:2}
0 & =& g_k^{\top}s^{Q}_k  -  (g_k^{\top}s_k^Q) \delta^{\NT}_k - \sigma_k \| s_k^Q \|^3_{M_k} \left(\delta^{\NT}_k\right)^2 ,
\end{eqnarray}
and thus, since $ \delta^{\NT}_k<0$ in this case,
 \begin{eqnarray*}
\delta^{\NT}_k& =& \frac{g_k^{\top}s_k^Q+ \sqrt{ \left(g_k^{\top}s_k^Q\right)^2 + 4 \sigma_k (g_k^{\top}s_k^Q)\| s_k^Q \|^3_{M_k} }}{ 2 \sigma_k \| s_k^Q \|^3_{M_k}} = \frac{2 }{1 -  \sqrt{1 + 4 \frac{\sigma_k \| s_k^Q \|^3_{M_k}}{g_k^{\top}s^Q_k} }}.
\end{eqnarray*}
From both cases, one deduces that
$ \delta^{\NT}_k=  \frac{2 }{1-\sign(g_k^{\top}s_k^Q) \sqrt{1 +4 \frac{\sigma_k \| s_k^Q \|^3_{M_k}}{|g_k^{\top}s^Q_k|} }} .$

Consider now the case where $s_k^Q$ is computed using an iterative subspace method. In this cas, one has $s^{\NT}_k= Q_k z^{\NT}_k$, where $z^{\NT}_k$ is the Newton-like step associated to the reduced subproblem (\ref{reduce_model_arc}). Hence by applying the first part of the proof (the exact case) to the reduced subproblem (\ref{reduce_model_arc}), it follows that
$$z^{\NT}_k  =   \bar{\delta}^{\NT}_k z^Q_k~~\mbox{ where~~}  \bar{\delta}^{\NT}_k=  \frac{2 }{1-\sign((Q_k^{\top} g_k)^{\top}z^Q_k) \sqrt{1 +4 \frac{\sigma_k \| z^Q_k \|^3_{{Q_k^{\top} M_k Q_k}}}{|(Q_k^{\top} g_k)^{\top}z^Q_k|} }}, $$ where $z^Q_k$ is a stationary point of the quadratic part of the minimized model in (\ref{reduce_model_arc}). Thus, by substituting $z^{\NT}_k$ in the formula $s^{\NT}_k= Q_k z^{\NT}_k$, one gets
\begin{eqnarray*}
s^{\NT}_k &= & Q_k \left(   \frac{2 }{1-\sign((Q_k^{\top} g_k)^{\top}z^Q_k) \sqrt{1 +4 \frac{\sigma_k \| z^Q_k \|^3_{{Q_k^{\top} M_k Q_k}}}{|Q_k^{\top} g_k)^{\top}z^Q_k|} }} z^Q_k\right) \\
&=& \frac{2}{1- \sign(g_k^{\top}Q_k z^Q_k)\sqrt{1 + 4 \frac{\sigma_k  \| Q_k z^Q_k \|^3_{{M}_k}}{|g_k^{\top}Q_k z^Q_k|}}}  Q_k z^Q_k \\
&=&  \frac{2}{1- \sign(g_k^{\top}s_k^Q)\sqrt{1 + 4 \frac{\sigma_k  \| s_k^Q\|^3_{{M}_k}}{|g_k^{\top}s_k^Q|}}}  s_k^Q.
\end{eqnarray*}
\end{proof}

In general, for ARC algorithm, the matrix $M_k$ can be any arbitrary SPD matrix. Our goal, in this section, is to determine how one can choose the matrix $M_k$ so that the  Newton-like step $s^{\NT}_k$  becomes a  stationary point of the subproblem (\ref{eq:nl_ARC_subproblem}). The following theorem gives explicitly the necessary and sufficient  condition on the matrix $M_k$ to reach this aim.
\begin{theorem} \label{th:2} Let Assumption \ref{asm:1} hold. The step  $s^{\NT}_k$ is a stationary point for the subproblem (\ref{eq:nl_ARC_subproblem}) if and only if  there exists $\theta_k >0$ such that $M_ks_k^Q= \frac{\theta_k}{g_k^{\top}s_k^Q}g_k$.
 Note that $\theta_k=\|s^Q_k\|^2_{M_k}$.
\end{theorem}
\begin{proof} 
Indeed, in the exact case, if we suppose that the step $ s^{\NT}_k$ is a stationary point of the subproblem~(\ref{eq:nl_ARC_subproblem}), this means that
 \begin{eqnarray}
 \label{equ:1}
 \nabla m_k(s^{\NT}_k) & = & g_k + B_ks^{\NT}_k+\sigma_k \|s^{\NT}_k\|_{M_k} M_k s^{\NT}_k= 0,
\end{eqnarray}
In  another hand, $s^{\NT}_k= \delta^{\NT}_k s^{Q}_k$ where $\delta^{\NT}_k$ is solution of $
g_k^{\top}s^{Q}_k  -  (g_k^{\top}s_k^Q)\delta^{\NT}_k +\sigma_k \| s_k^Q \|^3_{M_k} |\delta^{\NT}_k| \delta^{\NT}_k =0$ (such equation can be deduced from (\ref{equ:n:1}) and (\ref{equ:n:2})). Hence, we obtain that
\begin{eqnarray*}
0 & =& \nabla m_k(s^{\NT}_k) =g_k - \delta^{\NT}_k g_k +  \sigma_k   |\delta^{\NT}_k| \delta^{\NT}_k   \| s_k^Q \|_{M_k} M_k s^Q_k   \nonumber \\
&=& \left( 1 - \delta^{\NT}_k \right)g_k + \left( \frac{\sigma_k \| s_k^Q\|_{M_k}^3 }{g_k^{\top}s_k^Q}   |\delta^{\NT}_k| \delta^{\NT}_k \right) \left(  \frac{g_k^{\top}s_k^Q}{\| s^Q\|_{M_k}^2} M_k s_k^Q \right) \\
 &=&  \left( \frac{\sigma_k \| s_k^Q\|_{M_k}^3 }{g_k^{\top}s_k^Q}   |\delta^{\NT}_k| \delta^{\NT}_k \right)   \left(g_k -   \frac{g_k^{\top}s_k^Q}{\| s_k^Q\|_{M_k}^2} M_k s_k^Q \right).
\end{eqnarray*}
Equivalently, we conclude that
$
M_k s_k^Q= \frac{\theta_k}{g_k^{\top}s_k^Q}g_k
$ where $\theta_k =  \| s_k^Q\|_{M_k}^2 >0$. 
A similar proof applies when a subspace method is used to compute $s_k^Q$.
\end{proof}

The key condition to ensure that the ARC subproblem stationary point is equal to the Newton-like step $s^{\NT}_k$,
 is the choice of the matrix $M_k$ which satisfies the following secant-like equation $M_ks_k^Q= \frac{\theta_k}{g_k^{\top}s_k^Q}g_k$ for a given $\theta_k>0$.  The existence of such matrix $M_k$ is not problematic as far as Assumption \ref{asm:1} holds. In fact, Theorem \ref{th:equivnm} will explicit  a range of  $\theta_k>0$ for which the matrix $M_k$ exists. Note that in the formula of $s^{\NT}_k$, such matrix is used only through the computation of the $M_k$-norm of $s_k^Q$. Therefore an explicit formula of the matrix $M_k$ is not needed,  and only the value of $\theta_k = \|s_k^Q\|_{M_k}^2$ suffices for the computations. 

When the matrix $M_k$ satisfies the desired properties (as in Theorem \ref{th:2}), one is ensured that $ s^{\NT}_k$ is a stationary point for the model $m_k$. However, ARC algorithm imposes on the approximate step to satisfy the Cauchy decrease given by (\ref{Cauchy_decrease}), and such condition is not guaranteed by $s^{\NT}_k$ as the model $m_k$ may be non-convex.  In the next theorem, we show that for a sufficiently large $\sigma_k$, $ s^{\NT}_k$ is getting the global minimizer of $m_k$ and thus satisfying the Cauchy decrease is not an issue anymore. 
\begin{theorem} \label{cor:1} 
Let Assumption \ref{asm:1} hold. Let $M_k$ be  an SPD matrix which satisfies $M_ks_k^Q= \frac{\theta_k}{g_k^{\top}s_k^Q}g_k$ for a fixed $\theta_k >0$. 
If the matrix $Q_k^{\top}(B_k + \sigma_k \|s^{\NT}_k\|_{M_k} M_k)Q_k$ is positive definite, then the step $s^{\NT}_k$ is the unique minimizer of the subproblem (\ref{eq:nl_ARC_subproblem}) over the subspace $\mathcal{L}_k$. 
 \end{theorem}
 \begin{proof}
Indeed, when $s^Q$ is computed exactly (i.e., $Q_k=I_n$ and $\mathcal{L}_k=\real^n$), then using \cite[Theorem 3.1]{Conn_Gould_Toin_2000} one has that, for a given vector $s^*_k$, it is a global minimizer of $m_k$ if and only if it satisfies  
$$
(B_k + \lambda^*_k M_k)s^*_k=-g_k
$$
where $B_k +  \lambda^*_k M_k$ is positive semi-definite matrix and $\lambda^*_k= \sigma_k \|s^*_k\|_{M_k}$. Moreover, if $B_k +  \lambda^*_k M_k$ is positive definite, $s^*_k$ is unique.

Since $M_ks_k^Q= \frac{\theta_k}{g_k^{\top}s_k^Q}g_k$, by applying Theorem \ref{th:2}, we see that  
$$
(B_k + \lambda^{\NT}_k M_k)s^{\NT}_k=-g_k
$$
with $\lambda^{\NT}_k=\sigma_k \|s^{\NT}_k\|_{M_k}$. Thus, if we assume that $B_k + \lambda^{\NT}_k M_k$ is positive definite matrix, then $s^{\NT}_k$ is the unique global minimizer of the subproblem (\ref{eq:nl_ARC_subproblem}). 

Consider now, the case where $s^Q_k$ is computed using a subspace method, since $M_ks^Q_k= \frac{\theta_k}{g_k^{\top}s_k^Q}g_k$ one has that $Q_k^{\top}M_kQ_k z^Q_k= \frac{\theta_k}{(Q_k^{\top} g_k)^{\top}z^Q_k} Q_k^{\top}g_k$. 
Hence, if we suppose that the matrix $Q_k^{\top}(B_k + \lambda^{\NT}_k M_k)Q_k$ is positive definite, by applying the same proof of the exact case to the reduced subproblem (\ref{reduce_model_arc}), we see that the step $z^{\NT}_k$ is the unique global minimizer of the subproblem  (\ref{reduce_model_arc}). We conclude that $s^{\NT}_k= Q_k z^{\NT}_k$ is the global minimizer of the subproblem  (\ref{eq:nl_ARC_subproblem}) over the subspace $\mathcal{L}_k$. 
\end{proof}

Theorem \ref{cor:1} states that the step $s^{\NT}_k$ is the global minimizer of the cubic model $m_k$ over the subspace $\mathcal{L}_k$ as far as the matrix $Q_k^{\top}(B_k + \sigma_k \|s_k^{\NT}\|_{M_k} M_k)Q_k$ is positive definite, where $\lambda^{\NT}_k=\sigma_k \|s^{\NT}_k\|_{M_k}$. 
Note that  
\begin{eqnarray*}
\lambda^{\NT}_k &= & \sigma_k \|s^{\NT}_k\|_{M_k} = \frac{2 \sigma_k  \| s^Q_k \|_{M_k}}{\left |1-\sign(g_k^{\top}s_k^Q) \sqrt{1 +4 \frac{\sigma_k \| s_k^Q \|^3_{M_k}}{|g_k^{\top}s_k^Q|} }\right |} \to +\infty~~~~\mbox{as $\sigma_k \rightarrow \infty$.}
\end{eqnarray*}
Thus, since  $M_k$ is an SPD matrix and the regularization parameter $\sigma_k$ is increased for unsuccessful iterations in Algorithm~\ref{algo:ARC},  the positive definiteness of  matrix $Q_k^{\top}(B_k + \sigma_k \|s^{\NT}_k\|_{M_k} M_k)Q_k$ is guaranteed
 after finitely many unsuccessful iterations. 
In other words, one would have insurance that $ s^{\NT}_k$ will satisfy the Cauchy decrease after a certain number of unsuccessful iterations.
\section{Complexity Analysis of the LS-ARC Algorithm}
\label{section:3}

For the well definiteness of the algorithm LS-ARC, one needs first to show that the proposed $M_k$-norm is uniformly equivalent to the Euclidean norm.  The next theorem gives a range of choices for the parameter $\theta_k$ to ensure the existence of an SPD matrix $M_k$ such as $M_ks_k^Q= \frac{\theta_k}{g_k^{\top}s_k^Q }  g_k$ and the $M_k$-norm is uniformly equivalent to the $\ell_2$-norm. 

\begin{theorem} \label{th:equivnm}Let Assumption \ref{asm:1}  hold. If
\begin{equation} \label{I_def}
\begin{tabular}{l}
$ \theta_k = \beta_k \|s^Q_k\|^2 ~~\mbox{where}~ \beta_k \in ]\beta_{\min}, \beta_{\max}[ ~~\mbox{and}~~  \beta_{\max}> \beta_{\min}>0,$
\end{tabular}
\end{equation} then there exists an SPD matrix $M_k$ such as  

\begin{enumerate}[i)]
\item $M_ks_k^Q= \frac{\theta_k}{g_k^{\top}s_k^Q }  g_k$, 
\item the $M_k$-norm is uniformly equivalent to the $\ell_2$-norm on $\real^n$ and for all $x \in \real^n$, one has
\begin{equation} \label{eq:norms}
\frac{\sqrt{\beta_{\min}}}{\sqrt{2} } \|x\| \le \| x\|_{M_k} \le \frac{\sqrt{2 \beta_{\max} }}{\epsilon_d} \| x\|.
  \end{equation}
\item Moreover, one has  $\|s_k^Q\|_{M_k}^2 = \beta_k \|s^Q_k\|^2$ and  $\|g_k\|_{M_k}^2 = \chi_k \|g_k\|^2$ where
 $\chi_k= \beta_k \left(\frac{5}{2}  - \frac{3}{2} \cos(\varpi_k)^2 + 2 \left(\frac{1- \cos(\varpi_k)^2}{\cos(\varpi_k)}\right)^2\right)$ and $\cos(\varpi_k) = \frac{g_k^{\top}s_k^Q}{\|g_k\| \| s_k^Q\|}$
 \end{enumerate}
\end{theorem}
\begin{proof} Let $\bar{s}_k^Q= \frac{s_k^Q}{\|s_k^Q\|}$ and $\bar{g_k}$ be an orthonormal vector to $\bar{s}_k^Q$ (i.e., $\|\bar{g}_k \| =1$ and $\bar{g}_k^\top \bar{s}_k^Q=0$) such that
\begin{eqnarray} \label{eq:g_coordinate}
\frac{g_k}{\|g_k\|} = \cos(\varpi_k) \bar{s}_k^Q + \sin(\varpi_k) \bar{g_k}.
\end{eqnarray}

For a given $ \theta_k = \beta_k \|s^Q_k\|^2 ~~\mbox{where}~ \beta_k \in ]\beta_{\min}, \beta_{\max}[ ~~\mbox{and}~~  \beta_{\max}> \beta_{\min}>0,$ one would like to construct an SPD matrix $M_k$ such as $M_ks_k^Q= \frac{\theta_k g_k}{g_k^{\top}s_k^Q }$, hence
\begin{eqnarray*}
M_k\bar{s}_k^Q= \frac{\theta_k g_k}{g_k^{\top}s_k^Q \|s_k^Q\|}& =& \frac{\theta_k \|g_k\|}{g_k^{\top}s^Q_k \|s_k^Q\|} \left( \cos(\varpi_k) \bar{s}_k^Q + \sin(\varpi_k) \bar{g}_k  
\right)\\ 
&=& \beta_k \bar{s}_k^Q +\beta_k \tan (\varpi_k) \bar{g}_k. 
\end{eqnarray*}
Using the symmetric structure of the matrix $M_k$,  let $\gamma_k$ be a positive parameter such as
\begin{eqnarray*}
M_k= \left[\bar{s}_k^Q,   \bar{g}_k  \right]  N_k \left[\bar{s}_k^Q,   \bar{g}_k  \right]^{\top} \text{ where } N_k=
  \left[ {\begin{array}{cc}
   \beta_k & \beta_k \tan (\varpi_k) \\
   \beta_k \tan (\varpi_k)  & \gamma_k \\
  \end{array} } \right].
\end{eqnarray*}
The eigenvalues $ \lambda^{\min}_k$ and $\lambda^{\max}_k$ of the matrix $N_k$ are the roots of 
$$  \lambda^2 -  \left(\beta_k + \gamma_k \right) \lambda + \beta_k \gamma_k -   \left( \beta_k \tan (\varpi_k)   \right)^2=0,$$
hence 
\begin{eqnarray*}
 \lambda^{\min}_k = \frac{\left(\beta_k+ \gamma_k \right)-\sqrt{\vartheta_k}}{2}&\text{ and } &   
  \lambda^{\max}_k=\frac{\left(\beta_k + \gamma_k \right)+\sqrt{\vartheta_k}}{2}, 
\end{eqnarray*}
where $ {\vartheta_k} =  \left(\beta_k - \gamma_k \right)^2 + 4\left( \beta_k \tan (\varpi_k)   \right)^2.$ Note that both eigenvalues are monotonically increasing as functions of $\gamma_k$.

One may choose $\lambda^{\min}_k$ to be equal to $\frac{1}{2} \beta_k = \frac{1}{2} \frac{\theta_k}{\|s_k^Q\|^2} $, therefore  $\lambda^{\min}_k > \frac{1}{2}\beta_{\min}$ is uniformly bounded away from zero.  In this case, from the expression of $\lambda^{\min}_k$, we deduce that $\gamma_k =   2 \beta_k \tan(\varpi_k)^2 +\beta_k/2 $ and 
\begin{eqnarray*} \label{val:gamma}
\lambda^{\max}_k &= &\frac{3}{4} \beta_k + \beta_k \tan(\varpi_k)^2 +\sqrt{\frac{1}{16} \beta_k^2 + \frac{1}{2}\beta_k^2 \tan(\varpi_k)^2 + \beta_k^2 \tan(\varpi_k)^4} \\
 &=& \beta_k \left( \frac{3}{4} +  \tan(\varpi_k)^2 +\sqrt{\frac{1}{16}  + \frac{1}{2} \tan(\varpi_k)^2 + \tan(\varpi_k)^4}  \right)\\
 &=&  \beta_k \left( 1 +  2\tan(\varpi_k)^2\right).
\end{eqnarray*}
From Assumption \ref{asm:1}, i.e., $|g_k^{\top}s_k^Q| \ge \epsilon_d \|g_k \| \| s_k^Q\| $ where $\epsilon_d>0$, one has $\tan(\varpi_k)^2 \le  \frac{1- \epsilon_d^2}{\epsilon_d^2}.$ Hence, 
$$
\lambda^{\max}_k \le \beta_{\max} \left(  1 +  2\frac{1- \epsilon_d^2}{\epsilon_d^2} \right)
\le \frac{2 \beta_{\max}}{\epsilon_d^{2}} 
$$
A possible choice for the matrix $M_k$ can be obtained by completing the vectors family $\{\bar{s}_k^Q,\bar{g}_k \}$ to an orthonormal basis $\{\bar{s}_k^Q,\bar{g}_k, q_3, q_4, \ldots, q_n \}$ of $\real^n$ as follows:
$$
M_k =  [ \bar{s}_k^Q,\bar{g}_k, q_3, \ldots, q_n ]
  \left[ {\begin{array}{cc}
   N_k & 0 \\
   0 & D 
  \end{array} } \right]  [ \bar{s}_k^Q,\bar{g}_k, q_3, \ldots, q_n ]^{\top},
$$
where $D=\diag(d_3, \ldots, d_n) \in \real^{(n-2)\times (n-2)}$ with positive diagonal entrees independent from $k$ . 
One concludes that for all 
$ \theta_k = \beta_k \|s^Q_k\|^2 ~~\mbox{where}~ \beta_k \in ]\beta_{\min}, \beta_{\max}[ ~~\mbox{and}~~  \beta_{\max}> \beta_{\min}>0,$ the eigenvalue of the constructed $M_k$  are uniformly bounded away from zero and from above, hence 
the scaled $M_k$-norm is uniformly equivalent to the $\ell_2$-norm on $\real^n$ and for all $x \in \real^n$, one has
\begin{equation*} 
\frac{\sqrt{\beta_{\min}}}{\sqrt{2} } \|x\| \le \sqrt{\lambda^{\min}_k}  \|x\| \le \| x\|_{M_k} \le \sqrt{\lambda^{\max}_k}  \|x\| \le \frac{\sqrt{2 \beta_{\max} }}{\epsilon_d} \| x\|.
  \end{equation*}

By multiplying   $M_ks_k^Q= \frac{\theta_k}{g_k^{\top}s_k^Q }  g_k$ from both sides by $s_k^Q$, one gets $$\|s_k^Q\|_{M_k}^2 = \theta_k=\beta_k \|s^Q_k\|^2.$$

 Moreover, using (\ref{eq:g_coordinate}) and (\ref{val:gamma}), one has
 \begin{eqnarray*}
  \|g_k\|_{M_k}^2 &=& \|g_k\|^2 \left( \cos(\varpi_k) \bar{s}_k^Q + \sin(\varpi_k) \bar{g_k} \right)^{\top} \left(  \cos(\varpi_k) M_k\bar{s}_k^Q + \sin(\varpi_k) M_k\bar{g_k} \right)\\
  &=&  \|g_k\|^2 \left(   \frac{\theta_k \cos(\varpi_k)^2}{\|s_k^Q\|^2} + \gamma_k \sin(\varpi_k)^2 + 2 \sin(\varpi_k) \cos(\varpi_k)  \frac{\theta_k \tan(\varpi_k)^2}{\|s_k^Q\|^2}\right)\\
  &=&  \beta_k \|g_k\|^2  \left(   \cos(\varpi_k)^2 +\frac{5}{2} \sin(\varpi_k)^2 + 2 \sin(\varpi_k)^2 \tan(\varpi_k)^2 \right)\\
  &=&  \beta_k  \|g_k\|^2 \left( \frac{5}{2}  - \frac{3}{2} \cos(\varpi_k)^2 + 2 \left(\frac{1- \cos(\varpi_k)^2}{\cos(\varpi_k)}\right)^2 \right).
  \end{eqnarray*}
\end{proof}
 
A direct consequence of Theorem \ref{th:equivnm}  is that, by choosing $ \theta_k>0$ of the form $\beta_k \|s^Q_k\|^2$ where $\beta_k \in ]\beta_{\min}, \beta_{\max}[$ and $\beta_{\max}> \beta_{\min}>0$ during the application of LS-ARC algorithm, the global convergence and complexity bounds of LS-ARC algorithm can be derived straightforwardly from the classical ARC analysis \cite{CCartis_NIMGould_PhLToint_2011_b}. In fact, as far as the objective function $f$ is continuously differentiable, its gradient is Lipschitz continuous, and its approximated Hessian $B_k$ is bounded for all iterations (see \cite[Assumptions AF1, AF4, and AM1]{CCartis_NIMGould_PhLToint_2011_b}), the LS-ARC algorithm is globally convergent and will required at most a number of iterations of order $\epsilon^{-2}$ to produce a point $x_{\epsilon}$  with $\|\nabla f(x_{\epsilon})\| \le \epsilon$ \cite[Corollary 3.4]{CCartis_NIMGould_PhLToint_2011_b}. 
 
In what comes next, we assume the following on the objective function $f$:
\begin{assumption} \label{asm:f}
Assume that $f$ is twice  continuously differentiable with Lipschitz continuous Hessian, i.e., there exists a constant $L\ge 0$ such that for all $x, y\in \real^n$ one has
$$
\|\nabla^2 f(x) - \nabla^2 f(y)  \| \le L \|x -y \|.
$$
\end{assumption}
When the matrix $B_k$ is set to be equal to the exact Hessian of the problem and under Assumption \ref{asm:f}, one can improve the function-evaluation complexity to be  $\mathcal{O}(\epsilon^{-3/2})$ for ARC algorithm by imposing, in addition to the Cauchy decrease, another termination condition during the computation of the trial step $s_k$ (see  \cite{CCartis_NIMGould_PhLToint_2011_b,Birgin2016}). Such condition is of the form
\begin{eqnarray}
\label{TC_s}
\|\nabla m_k(s_k) \| & \le & \zeta \|s_k \|^2,
\end{eqnarray}
where $\zeta> 0$ is a given constant chosen at the start of the algorithm. 

When only an approximation of the Hessian is available during the application Algorithm \ref{algo:ARC}, an additional condition has to be imposed on the Hessian approximation $B_k$  in order to ensure an optimal complexity of order $\epsilon^{-3/2}$. Such condition is often considered as (see \cite[Assumption AM.4]{CCartis_NIMGould_PhLToint_2011_b}):
\begin{assumption}\label{asm:B} The matrix $B_k$ approximate the Hessian $\nabla^2 f(x_k)$ in the sense that
\begin{eqnarray} \label{strong:Dennis-more:approxiation}
\|(\nabla^2 f(x_k) -B_k )s_k\| \le C \|s_k\|^2
\end{eqnarray}
for all $k\ge 0$ and for some constant $C>0$. 
\end{assumption}

Similarly, for LS-ARC algorithm, the complexity bound can be improved to be of the order of $\epsilon^{-3/2}$ if one includes the two following requirements (a) the $s_{k}$ satisfies the criterion condition~(\ref{TC_s}) and (b) the Hessian approximation matrix $B_k$ has to satisfy Assumption \ref{asm:B}. When our proposed $M_k$-norm is used, the termination condition (\ref{TC_s}) imposed on the cubic model $m_k$ can be expressed only in terms of $s_k^Q$ and $\nabla m_k^Q$. The latter condition will be required in the LS-ARC algorithm at each iteration to ensure that it takes at most $\mathcal{O}(\epsilon^{-3/2})$  iterations to reduce the gradient norm below $\epsilon$. Such result is given in the following proposition 
\begin{proposition} 
\label{lm1:wcc} Let Assumption \ref{asm:1}  hold. Let $ \theta_k = \beta_k \|s^Q_k\|^2$ where $\beta_k \in ]\beta_{\min}, \beta_{\max}[$ and $\beta_{\max}> \beta_{\min}>0$. Then imposing the condition (\ref{TC_s}) in Algorithm \ref{algo:LS-ARC} is equivalent to the following condition 
\begin{eqnarray}
\label{TC_s_arc_en}
\|\nabla m_k^Q(s_k^Q) \| & \le & 
\frac{2 \sign(g_k^{\top}s_k^Q) \zeta  }{-1+\sign(g_k^{\top}s_k^Q) \sqrt{1 +4 \frac{\sigma_k \theta_k^{3/2}}{|g_k^{\top}s^Q_k|} }} \|s_k^Q\|^2.
\end{eqnarray}
\end{proposition}
\begin{proof} Since Assumption \ref{asm:1} holds and $ \theta_k = \beta_k \|s^Q_k\|^2$ as in (\ref{I_def}), Theorem \ref{th:equivnm} implies the existence of an SPD matrix $M_k$ such that $M_ks_k^Q=\frac{\theta_k}{g_k^{\top}s_k^Q}g_k $. Using such $M_k$-norm, an approximate solution of the cubic model $m_k$ is of the form $s_k= \delta_k s^{Q}_k$ where $\delta_k$ is solution of $
g_k^{\top}s^{Q}_k  -  (g_k^{\top}s_k^Q)\delta_k +\sigma_k \| s_k^Q \|^3_{M_k} |\delta_k| \delta_k =0$. Hence,
\begin{eqnarray*}
\nabla m_k(s_k) & = & g_k + B_k s_k+ \sigma_k \|s_k \|_{M_k} M_k s_k  \\
                         &=& g_k + \delta_k B_k s^Q_k +  \sigma_k   |\delta_k| \delta_k   \| s_k^Q \|_{M_k} M_k s^Q_k.     \end{eqnarray*}
Since $M_ks_k^Q=\frac{\theta_k}{g_k^{\top}s_k^Q}g_k $ with $\theta_k=\|s_k^Q\|_{M_k}^2$, one has
          \begin{eqnarray*}                                    
     \nabla m_k(s_k)                             & =&  g _k+ \delta_k B_k s_k^Q  +  \frac{\sigma_k   |\delta_k| \delta_k   \| s_k^Q \|_{M_k}^3}{g_k^{\top}s_k^Q} g_k \\
     &=& \left(1 + \frac{\sigma_k   |\delta_k| \delta_k   \| s_k^Q \|_{M_k}^3}{g_k^{\top}s_k^Q}  \right) g_k + \delta_k B_k s_k^Q .
\end{eqnarray*}
From the fact that
$
g_k^{\top}s^{Q}_k  -  (g_k^{\top}s_k^Q)\delta_k +\sigma_k \| s_k^Q \|^3_{M_k} |\delta_k| \delta_k =0$, one deduces
     \begin{eqnarray*}                                    
     \nabla m_k(s_k)                             & =& \delta_k \left(g_k +  B_k s_k^{Q} \right)=  \delta_k  \nabla m_k^Q(s_k^{Q}).
\end{eqnarray*}
Hence, the condition (\ref{TC_s}) is being equivalent to 
\begin{eqnarray*}
\|\nabla m_k^Q(s_k^{Q}) \| & \le & \frac{ \zeta}{\delta_k}  \|s_k \|^2 = \zeta |\delta_k| \|s_k^Q\|^2.
\end{eqnarray*}
\end{proof}
 
We note that the use of an exact solver to compute $s^Q_k$ implies that the condition (\ref{TC_s_arc_en}) will be automatically satisfied for a such iteration.
Moreover, when a subspace method is used to approximate the step $s^Q_k$, we note the important freedom to add an additional preconditioner to the problem. In this case, one would solve the quadratic problem with  preconditioning until  the criterion~(\ref{TC_s_arc_en}) is met. This is expected to happen early along the Krylov iterations when the preconditioner for the linear system $B_ks=-g_k$ is good enough.

 Algorithm \ref{algo:LS-ARC:s} summarized a second-order variant of LS-ARC, referred here as $\mbox{LS-ARC}_{\mbox{(s)}}$, which is guaranteed to have an improved iteration worst-case complexity of order $\epsilon^{-3/2}$ (see Theorem \ref{th:4-ters}).
 \noindent
\begin{algorithm}[!ht]
\DontPrintSemicolon
\SetAlgoNlRelativeSize{0}
\caption{\bf $\mbox{LS-ARC}_{\mbox{(s)}}$ Algorithm.}
\label{algo:LS-ARC:s}
\begin{rm}
\begin{description}
\item
In each iteration $k$ of Algorithm \ref{algo:LS-ARC}: 

Let $s^Q_k$ be an approximate solution of $B_k s=-g_k$ such as $|g_k^\top s^Q_k| \ge  \epsilon_d  \|g_k\|\|s^Q_k\|$ and the termination condition
\begin{eqnarray*}
\|\nabla m_k^Q(s_k^Q) \| & \le & 
\frac{2 \sign(g_k^{\top}s_k^Q) \zeta  }{-1+\sign(g_k^{\top}s_k^Q) \sqrt{1 +4 \frac{\sigma_k \beta_k^{3/2} \|s^Q_k\|^3}{|g_k^{\top}s^Q_k|} }} \|s_k^Q\|^2
\end{eqnarray*}
is satisfied for a given constant $\zeta >0$ chosen at the beginning of the algorithm.
\end{description}
\end{rm}
\end{algorithm} 

The Hessian approximation $B_k$ (as required by Assumption \ref{asm:B}) involves $s_k$, hence finding a new matrix $B_k$ so that the new $s_k^Q$ satisfies Assumption \ref{asm:1} using regularization techniques is not trivial as $s_k$ is unknown at this stage. 
A possible way to satisfy Assumption \ref{asm:B} without modifying $B_k$ is by choosing $s^Q_k$ as the first iterate to satisfy  (\ref{TC_s_arc_en}) when using a subspace method to solve the linear system $B_k s=-g_k$. Then, one checks if $s^Q_k$ satisfies Assumption \ref{asm:1} or not. If $s^Q_k$ violates the latter assumption, one runs further iterations of the subspace method until Assumption \ref{asm:1} will be satisfied. If the subspace method ends and Assumption \ref{asm:1} is still violated, one would restore such assumption by minimizing the cubic model using the $\ell_2$-norm until a successful outer iteration is found. 

For the sake of illustration, consider the minimization of the objective function $f(x,y) = x^2 - y^2$ for all $(x,y) \in \real^2$, starting from $x_0=(1,1)$, with $\sigma_0=1$, and $B_k$ being the exact Hessian of the problem during the application of the algorithm. One starts by checking if $s^Q_0 = (-1,-1)$ (i.e., the exact solution of the linear system $B_0 s = - g_0$) is a sufficient descent direction or not. Since the slop $g_0^\top s^Q_0$ is equal to zero for this example, the algorithm has to switch to the $\ell_2$-norm and thus $s_0$ will be set as a minimizer of the cubic model with the $\ell_2$-norm to define the cubic regularization term.  
Using a subproblem solver (in our case the \texttt{GLRT} solver from \texttt{GALAHAD} \cite{NIMGould_DOrban_PhLToint_2003}, more details are given in Section 6), one finds the step $s_0 = (-0.4220,  2.7063)$ and the point $x_1 = x_0 + s_0$ is accepted (with $f(x_1)=-13.4027$). 
Computing the new gradient $g_1 = (1.1559, -7.4126)$ and the quasi-Newton direction $s^Q_1 = (-0.5780, -3.7063)$, one has  $|g_1^\top s^Q_1| = 20.5485 \ge  \epsilon_d \|g_1\|\|s_1^Q\|=0.0205$ where $\epsilon_d=10^{-3}$ (hence Assumption 3.1 holds). We perform then our proposed LS strategy along the direction $s_1^Q$ to obtain the step $s_1$.
For this example, except the first one, all the remaining iterations $k$ satisfy the condition $|g_k^{\top} s^Q_k| \ge  \epsilon_d\|g_k\| \|s_k^Q\|$. 
We note that the regarded minimization problem is unbounded from below, hence $f$ decreases to $-\infty$ during the application of the algorithm.

LS strategies require in general to have a sufficient descent direction for each iteration, it seems then natural that one may need to choose $\epsilon_d$ to be large (close to 1) to target good performance. However, during the application of $\mbox{LS-ARC}_{\mbox{(s)}}$ and to satisfy Assumption \ref{asm:1}  (without modifying the matrix $B_k$), one may be encouraged to use an $\epsilon_d$ small. In what comes next, we will give a detailed complexity analysis of the $\mbox{LS-ARC}_{\mbox{(s)}}$ algorithm in this case. In particular, we will explicit how the complexity bound will depend on the choice of the constant $\epsilon_d$.
The following results are obtained from \cite[Lemma 2.1]{Birgin2016} and \cite[Lemma 2.2]{Birgin2016}:
\begin{lemma} 
Let Assumptions \ref{asm:f} and \ref{asm:1} hold and consider Algorithm \ref{algo:LS-ARC:s}. 
Then for all $k\ge 0$, one has
\begin{eqnarray} \label{decrease:cubic}
f(x_k) - m^Q_k(s_k) & \ge&\frac{\sigma_k}{3} \|s_k\|^3_{M_k}, 
\end{eqnarray}
and 
\begin{eqnarray} \label{sigma:max}
\sigma_k & \le& \sigma_{\max} := \max \left\{ \sigma_0, \frac{3 \nu_2 L}{2 (1-\eta) } \right\}.
\end{eqnarray}
\end{lemma}
 The next lemma is an adaptation of \cite[Lemma 2.3]{Birgin2016} when the proposed $M_k$-norm is used in the ARC framework.
\begin{lemma}
Let Assumptions \ref{asm:f},  \ref{asm:B} and \ref{asm:1} hold. Consider Algorithm \ref{algo:LS-ARC:s} with $ \theta_k = \beta_k \|s^Q_k\|^2$ where $\beta_k \in ]\beta_{\min}, \beta_{\max}[$ and $\beta_{\max}> \beta_{\min}>0$. 
Then for all $k\ge 0$
\begin{eqnarray*} \label{eq:g_s}
\| s_k\| & \ge & \left( \frac{\|g_{k+1}\|}{L+ C+2\sqrt{2} \sigma_{\max} \beta_{\max}^{3/2}\epsilon_d^{-3} + \zeta }\right)^{\frac{1}{2}}.
\end{eqnarray*}
\end{lemma}
\begin{proof}Indeed, using Assumptions \ref{asm:f} and \ref{asm:B} within Taylor expansion, one has
\begin{eqnarray*}
\| g_{k+1}\| & \le & \| g_{k+1} - \nabla m_k (s_k) \| + \| \nabla m_k (s_k) \| \\
& \le & \|g_{k+1} - g_k -B_k s_k - \sigma_k \|s_k\|_{M_k} M_k s_k \| + \zeta \|  s_k\|^2 \\
& \le & \|g_{k+1} - g_k - \nabla^2 f(x_k) s_k\| + \| (\nabla^2 f(x_k)-B_k) s_k \| + \sigma_k \|M_k\|^{3/2} \|s_k\|^2 + \zeta \|  s_k\|^2\\
& \le & L\| s_k \|^2 +  C \| s_k \|^2+(\sigma_{k} \|M_k\|^{3/2} +\zeta) \| s_k \|^2.
\end{eqnarray*}
Using (\ref{sigma:max}), one has
\begin{eqnarray*}
\| g_{k+1}\|& \le & (L +  C +\sigma_{\max} \|M_k\|^{3/2} +\zeta) \| s_k \|^2.
\end{eqnarray*}
Since Assumption \ref{asm:B} holds and $ \theta_k = \beta_k \|s^Q_k\|^2$ where $\beta_k \in ]\beta_{\min}, \beta_{\max}[$, then using Theorem \ref{th:equivnm}, the matrix $M_k$ norm is bounded from above by $2\beta_{\max} \epsilon_d^{-2}$. Hence,
\begin{eqnarray*} 
\| g_{k+1}\| & \le & \left(L+ C+ 2\sqrt{2} \sigma_{\max} \beta_{\max}^{3/2} \epsilon_d^{-3} +\zeta \right) \| s_k \|^2.
\end{eqnarray*}
\end{proof}
\begin{theorem} \label{th:4-ters} 
Let Assumptions \ref{asm:f},  \ref{asm:B} and \ref{asm:1} hold. Consider Algorithm \ref{algo:LS-ARC:s} with $ \theta_k = \beta_k \|s^Q_k\|^2$ where $\beta_k \in ]\beta_{\min}, \beta_{\max}[$ and $\beta_{\max}> \beta_{\min}>0$. 
Then, given an $\epsilon>0$, Algorithm \ref{algo:LS-ARC:s} needs at most 
$$
\left \lfloor \kappa_{\mbox{s}}(\epsilon_d) \frac{f(x_0) - f_{\mbox{low}}}{\epsilon^{3/2}} \right \rfloor
$$
iterations to produce an iterate $x_{\epsilon}$ such that $\|\nabla f(x_{\epsilon})  \| \le \epsilon$ where $f_{\mbox{low}}$ is a lower bound on $f$ and $\kappa_{\mbox{s}}(\epsilon_d)$ is given by 
\begin{eqnarray*}
\kappa_{\mbox{s}}(\epsilon_d) & =& \frac{6\sqrt{2}\left(L+ C+ 2\sqrt{2} \sigma_{\max} \beta_{\max}^{3/2} \epsilon_d^{-3} + \zeta \right)^{3/2}}{\eta \sigma_{\min}\beta_{\min}^{3/2}}.
\end{eqnarray*}
 \end{theorem} 
 \begin{proof}
Indeed, at each iteration of Algorithm \ref{algo:LS-ARC:s}, one has
\begin{eqnarray*}
 f(x_k) - f(x_k + s_k)  & \ge & \eta (f(x_k) -m^Q_k(s_k)) \\
                                  & \ge & \frac{\eta \sigma_k}{3} \|s_k\|^3_{M_k} \\
                                  & \ge & \frac{\eta \sigma_{\min} \beta_{\min}^{3/2}}{6\sqrt{2}}  \|s_k\|^3 \\
                                  & \ge & \frac{\eta \sigma_{\min} \beta_{\min}^{3/2}}{6\sqrt{2}\left(L +C+ 2\sqrt{2} \sigma_{\max} \beta_{\max}^{3/2} \epsilon_d^{-3} +\zeta \right)^{3/2}} \|g_{k+1}\|^{3/2} \\
                                  & \ge & \frac{\eta \sigma_{\min} \beta_{\min}^{3/2}}{6\sqrt{2}\left(L + C+2\sqrt{2} \sigma_{\max} \beta_{\max}^{3/2} \epsilon_d^{-3}+\zeta \right)^{3/2}} \epsilon^{3/2}, \\
\end{eqnarray*}
by using (\ref{sdcond:arc}), (\ref{decrease:cubic}), (\ref{eq:norms}), (\ref{eq:g_s}), and the fact that $\|g_{k+1}\| \ge \epsilon$ before termination. Thus we deduce for all iterations as long as the stopping criterion does not occur
\begin{eqnarray*}
 f(x_{0}) - f(x_{k+1})  &= & \sum_{j=0}^{k} f(x_{j}) - f(x_{j+1}) \\
                                 &\ge& (k+1) \frac{\eta \sigma_{\min} \sqrt{\beta_{\min}}}{3\sqrt{2}\left(L +C+ 2\sqrt{2} \sigma_{\max} \beta_{\max}^{3/2} \epsilon_d^{-3}+\zeta \right)^{3/2}} \epsilon^{3/2}.
 \end{eqnarray*}
 Hence, the required number of iterations to produce an iterate $x_{\epsilon}$ such that $\|\nabla f(x_{\epsilon})  \| \le \epsilon$ is given as follow
 \begin{eqnarray*}
k+1  &\le  & \frac{6\sqrt{2}\left(L + C+ 2\sqrt{2} \sigma_{\max} \beta_{\max}^{3/2} \epsilon_d^{-3} +\zeta \right)^{3/2}}{\eta \sigma_{\min}\beta_{\min}^{3/2}} \frac{ f(x_{0}) - f_{\mbox{low}}}{\epsilon^{3/2} }
 \end{eqnarray*}
where $f_{\mbox{low}}$ is a lower bound on $f$. Thus the proof is completed.
 \end{proof}
 We note that $\kappa_{\mbox{s}}(\epsilon_d)$ can be large for small values of $\epsilon_d$. Hence, although the displayed worst-case complexity bound is of order $\epsilon^{-3/2}$, the latter can be worse than it appears to be if the value of $\epsilon_d$ is very small (i.e., the chosen direction is almost orthogonal with the gradient). Such result seems coherent regarding the LS algorithm strategy where it is required to have a sufficient descent direction (i.e., an $\epsilon_d$ sufficiently large).
\section{TR Algorithm Using a Specific $M_k$-Norm}
\label{section:4}

Similarly to ARC algorithm, it is possible to render TR algorithm behaves as an LS algorithm using the same scaled norm to define the trust-region neighborhood. 
As a reminder  in a basic TR algorithm \cite{Conn_Gould_Toin_2000}, one computes a trial step $p_k$ by approximately solving
 \begin{eqnarray}
\label{eq:nl_TR_subproblem}
\displaystyle \min_{p \in \real^n} &    m_k^Q(p) &  ~~~\mbox{s. t.}~~~ \|p\|_{M_k}  \le \Delta_k,
\end{eqnarray}
where $\Delta_k >0$ is known as the TR radius. As in ARC algorithms, the scaled norm $\|.\|_{M_k}$ may vary along the iterations and  $M_k$ is an SPD matrix. 

Once the trial step $p_k$ is determined, the objective function is  computed at $x_k + p_k$ and compared with the value predicted by the model at this point. If the model value predicts sufficiently well  the objective function (i.e., the iteration is \textit{successful}), the trial point $x_k + p_k$ will be accepted and the TR radius is eventually expanded (i.e., $\Delta_{k+1}=\tau_2 \Delta_{k}$ with $\tau_2\ge 1$). If the model turns out to predict poorly the objective function (i.e., the iteration is \textit{unsuccessful}), the trial point is rejected and the TR radius is contracted  (i.e., $\Delta_{k+1}=\tau_1 \Delta_{k}$ with $\tau_1 < 1$).  The ratio between the actual reduction and the predicted reduction  for the TR algorithms is defined as in ARC (see (\ref{rho})).
For a given scalar $0 < \eta < 1$, the iteration will be said \textit{successful} if $\rho_k \ge \eta$, and  \textit{unsuccessful} otherwise. Algorithm~\ref{algo:TR} gives a detailed description of a basic TR algorithm.

 \LinesNumberedHidden
\begin{algorithm}[!ht]
\SetAlgoNlRelativeSize{0}
\caption{\bf TR algorithm.}
\label{algo:TR}
\SetAlgoLined
\KwData{select an initial point $x_0$ and $0< \eta <1$.  Set the initial TR radius $\Delta_0>0$, the constants  $0 \le \tau_1< 1 \le \tau_2$, and $\Delta_{\max}>\Delta_0$.}
\For{$k= 1,  2, \ldots$}{
Compute the step $p_k$ as an approximate solution of (\ref{eq:nl_TR_subproblem}) such that
\begin{eqnarray}
\label{TR:Cauchy_decrease}
m_k^{Q}(p_k) &\le & m_k(p^{\cc}_{k})
\end{eqnarray}
where 
$p^{\cc}_{k} = - \alpha^{\cc}_{k} g_k \text{ and } \alpha^{\cc}_{k} = \displaystyle \arg \min_{0< t\le \frac{\Delta_k}{\|g_k\|_{M_k}}} m_k^Q(-t g_k)$\;
\eIf{$\rho_k \ge \eta$}{Set 
$x_{k+1}= x_k +p_k$ and $ \Delta_{k+1} =\min\{\tau_2 \Delta_k,\Delta_{\max}\}$; }{Set $x_{k+1} = x_k$ and $\Delta_{k+1}= \tau_1 \Delta_k$\;}
}
\end{algorithm}
Note that for the TR subproblem, the solution we are looking for lies either interior to the trust region, that is $\|p_k\|_{M_k} < \Delta_k$, or on the boundary, $\|p_k\|_{M_k} = \Delta_k$.  If the solution is interior, the solution $p_k$ is the unconstrained minimizer of the quadratic model $m_k^Q$. Such scenario can only happen if $m_k^Q$ is convex. In the non convex case a solution lies on the boundary of the trust region, while in the convex case a solution may or may not do so. Consequently in practice, the TR algorithm finds first the unconstrained minimizer of the model $m_k^Q$. If the model is unbounded from below, or if the unconstrained minimizer lies outside the trust region, the minimizer then occurs on the boundary of the trust region.

In this section, we will assume that the approximated solution $s_k^Q$ of the linear system $B_k s=-g_k$ is computed exactly.  
Using similar arguments as for ARC algorithm, one can extend the obtained results when a truncated step is used in the TR algorithm. 
Under Assumption \ref{asm:1}, we will call the Newton-like step associated with the TR subproblem the vector of the following form
 \begin{eqnarray}
\label{Newton_step_TR}
p^{\NT}_k =  \alpha_k^{\NT} s^Q_k,  ~~\mbox{where} ~\alpha^{\NT}_k   = \displaystyle \arg \min_{\alpha \in \mathcal{R}_k} m^Q_k(\alpha s_k^{Q}),
\end{eqnarray} 
where $\mathcal{R}_k =]0, \frac{\Delta_k}{ \|s_k^{Q}\|_{M_k}}]$ if  $g_k^{\top}  s_k^{Q} < 0$ and $\mathcal{R}_k=[ -\frac{\Delta_k}{ \|s_k^{Q}\|_{M_k}},0[ $ otherwise.

Similarly to ARC algorithm, one has the following results:
\begin{theorem} \label{th:1:tr}   Let Assumption \ref{asm:1} hold. 
\begin{enumerate}
\item 
The Newton-like step (\ref{Newton_step_TR}) is of the following form:
  \begin{eqnarray}
 \label{eq:TR_newton-step}
p^{\NT}_k & = & \alpha^{\NT}_k s^Q_k, ~~\mbox{where}~~  \alpha^{\NT}_k =  \min \left\{1, - \sign(g_k^{\top}s_k^Q) \frac{ \Delta }{\| s_k^Q\|_{M_k}}\right\}.
\end{eqnarray}
\item
When it lies on the border of the trust region  $p^{\NT}_k$  is a stationary point of the subproblem (\ref{eq:nl_TR_subproblem})  if and only if  $M_ks_k^Q= \frac{\theta}{g_k^{\top}s_k^Q}g_k$ where $\theta_k=\|s^Q_k\|^2_{M_k}$.
\item  Let $\lambda^{\NT}_k = \frac{g_k^{\top}s_k^Q}{\theta_k} \left(  1 + \sign(g_k^{\top}s_k^Q)  \frac{ \|s^Q_k\|_{M_k} }{\Delta_k} \right) $ and assume that $p^{\NT}_k$ lies on the border of the trust-region. Then, if the matrix $B_k + \lambda^{\NT}_k M_k$ is positive definite, the step $p^{\NT}_k$ will be the unique minimizer of the subproblem (\ref{eq:nl_TR_subproblem}) over the subspace $\mathcal{L}_k$.
\end{enumerate}
\end{theorem}
\begin{proof}
1.
To calculate the Newton-like step $p^{\NT}_k $, we first note, for all $\alpha \in \mathcal{R}_k$
 \begin{eqnarray}
 \label{min_N_TR}
m^Q_k(\alpha s^{Q}_k) -  m^Q_k(0) &=&  \alpha g_k^{\top}s^{Q}_k + \frac{\alpha^2}{2} [s^{Q}_k]^{\top} B_k [s^{Q}_k]  \nonumber \\ 
&=&  (g_k^{\top}s^{Q}_k)   \alpha  -  (g_k^{\top}s_k^Q) \frac{\alpha^2}{2}.
\end{eqnarray}

Consider the case where the curvature model along the Newton direction is positive, that is when $g_k^{\top}s_k^Q<0,$ (i.e., $\mathcal{R}_k =]0, \frac{\Delta_k}{ \|s_k^{Q}\|_{M_k}}]$)
and compute the value of the parameter $\alpha$ at which the unique minimizer of (\ref{min_N_TR}) is attained. Let $\alpha^*_k$ denotes this optimal parameter. 
Taking the derivative of (\ref{min_N_TR})  with respect to $\alpha$ and equating the result to zero, one has  $\alpha^*_k  = 1 $.
Two sub-cases may then occur. The first is when this minimizer lies within the trust region (i.e., $\alpha^*_k \|s^{Q}_k\|_{M_k} \le \Delta_k$), then 
 \begin{eqnarray*}
 \label{alpha_tr_N_1_1}
 \alpha^{\NT}_k &=& 1.
\end{eqnarray*}
If  $\alpha^*_k \|s^{Q}_k\|_{M_k} > \Delta_k$, then the line minimizer is outside the trust region and we have that 
 \begin{eqnarray*}
 \label{alpha_tr_N_1_2}
 \alpha^{\NT}_k  & = & \frac{\Delta_k}{\|s^{Q}_k\|_{M_k}}.
\end{eqnarray*}

Finally, we consider the case where the curvature of the model along the Newton-like step is negative, that is, when $g_k^{\top}s_k^Q> 0$. In that case, the minimizer lies on the boundary of the trust region, and thus 
 \begin{eqnarray*}
 \label{alpha_tr_N_2}
 \alpha^{\NT}_k  & = & - \frac{\Delta_k}{\|s^{Q}_k\|_{M_k}}.
\end{eqnarray*}
By combining all cases, one concludes that 
  \begin{eqnarray*}
p^{\NT}_k  & = &  \alpha^{\NT}_k  s^Q_k, ~~\mbox{where}~~   \alpha^{\NT}_k =  \min \left \{1, - \sign(g_k^{\top}s_k^Q) \frac{ \Delta_k }{\| s^Q_k\|_{M_k}}\right\}.
\end{eqnarray*}

2. Suppose that the Newton-like step lies on the border of the trust region, i.e., $  p^{\NT}_k =   \alpha^{\NT}_k  s^Q_k= - \sign(g_k^{\top}s_k^Q) \frac{ \Delta_k }{\| s^Q_k\|_{M_k}}s^Q_k.
$
The latter step is a stationary point of the subproblem (\ref{eq:nl_TR_subproblem}) if and only if 
there exists a Lagrange multiplier $\lambda^{\NT}_k\ge 0$ such that
\begin{eqnarray*}
(B_k+ \lambda^{\NT}_k M_k)p^{\NT}_k  &=  &-g_k.
\end{eqnarray*}
Substituting $p^{\NT}_k= \alpha^{\NT}_k s_k^Q$ in the latter equation, one has
\begin{eqnarray}
\label{eq:tr:1}
\lambda^{\NT}_k M_k s_k^Q  &=  &\left(1 - \frac{1}{\alpha^{\NT}_k}\right)g_k. 
\end{eqnarray}
By multiplying it from left by $(s^Q)^{\top}$, we deduce that 
\begin{eqnarray*} \label{lagrange-multiplier-tr}
\lambda^{\NT}_k &= & \left(1 - \frac{1}{\alpha^{\NT}_k}\right)\frac{g_k^{\top}s_k^Q}{\|s^Q_k\|^2_{M_k}}=\frac{g_k^{\top}s_k^Q}{\theta} \left(  1 + \sign(g_k^{\top}s_k^Q)  \frac{ \|s^Q_k\|_{M_k} }{\Delta_k} \right).\end{eqnarray*} 
By replacing the value of $\lambda^{\NT}_k$ in (\ref{eq:tr:1}), we obtain that $
M_ks_k^Q= \frac{\theta_k}{g_k^{\top}s_k^Q}g_k
$ where $\theta_k =  \| s_k^Q\|_{M_k}^2 >0$.

3. Indeed, when the step $p^{\NT}_k$ lies on the boundary of the trust-region and $M_k s_k^Q= \frac{\theta_k}{g_k^{\top}s_k^Q}g$. Then applying item (1) of Theorem \ref{th:1:tr}, we see that
$$
(B_k + \lambda^{\NT}_k M_k)p^{\NT}_k=-g_k
$$
with $\lambda^{\NT}_k=\frac{g_k^{\top}s_k^Q}{\theta_k} \left(  1 + \sign(g_k^{\top}s_k^Q)  \frac{ \|s^Q_k\|_{M_k} }{\Delta_k} \right)>0$. 
Applying \cite[Theorem 7.4.1]{Conn_Gould_Toin_2000}, we see that if we assume that the matrix $B_k + \lambda^{\NT}_k M_k$ is positive definite, then $p^{\NT}_k$ is the unique  minimizer of the subproblem (\ref{eq:nl_TR_subproblem}). 
\end{proof}

Given an SPD matrix $M_k$ that satisfies the secant equation  $ M_ks^Q_k= \frac{\theta_k}{g_k^{\top}s_k^Q} g_k$,  item (3) of Theorem \ref{th:1:tr} states that the step $p^{\NT}_k$ is the global minimizer of the associated subproblems over the subspace $\mathcal{L}_k$ as far as the matrix  $B_k + \lambda^{\NT}_k M_k$ is SPD. We note that $\lambda^{\NT}_k$ goes to infinity as the trust region radius $\Delta_k$ goes to zero, meaning that the matrix $B_k + \lambda^{\NT}_k M_k$ will be SPD as far as $\Delta_k$ is chosen to be sufficiently small. Since the TR update mechanism allows to shrink the value of $\Delta_k$ (when the iteration is declared as unsuccessful), satisfying the targeted condition will be geared automatically by the TR algorithm. 

Again, when Assumption \ref{asm:1} holds, we note that  \textit{unsuccessful} iterations in the TR Algorithm require only  updating the value of the TR radius $\Delta_k$ and the current step direction is kept unchanged.  For such iterations, as far as there exists a matrix $M_k$ such as $M_ks^Q_k= \frac{\beta_k \|s^Q_k\|^2}{g_k^{\top}s_k^Q}g_k$ where $\beta_k \in ]\beta_{\min}, \beta_{\max}[$ and $\beta_{\max}> \beta_{\min}>0$, the approximate solution of the TR subproblem  is obtained only by updating the step-size $\alpha^{\NT}_k$.  This means that the computational cost of unsuccessful iterations do not requires solving any extra subproblem. We note that during the application of the algorithm, we will take $\theta_k$ of the form $\beta_k \|s^Q_k\|^2_2$, where $\beta_k \in ]\beta_{\min}, \beta_{\max}[$ and $0<\beta_{\min}< \beta_{\max}$. Such choice of the parameter  $\theta_k$ ensures that the proposed $M_k$-norm uniformly equivalent to the $l_2$ one along the iterations (see Theorem \ref{th:equivnm}).

In this setting, TR algorithms behaves as an LS method with a specific backtracking strategy. In fact, at the $k^{\mbox{th}}$ iteration, the step is of the form $p_k= \alpha_k s^Q_k  $ where $s_k^Q $ is the (approximate) solution of the linear system $B_k s=-g_k$. 
 The step length $ \alpha_k>0$ is chosen such as
\begin{eqnarray} \label{sdcond}
\frac{f(x_k) - f(x_k +p_{k} )}{ f(x_k) - m_k^Q(p_{k})} \ge  \eta  & ~~\mbox{and}~& m^Q_k(p_{k}) \le m^Q_k(-\alpha^{\cc}_k g_k).
\end{eqnarray}
The  values of $\alpha_k$ and $\alpha^{\cc}_k$ are computed respectively as follows:
\begin{eqnarray} \label{eq:deltatr:k}
       \alpha_k &=    & \min \left\{1, - \sign(g_k^{\top}s^Q_k) \frac{ \Delta_k }{\beta_k^{1/2}\|s^Q_k\|}\right\} 
\end{eqnarray}
and
  \begin{eqnarray}
 \label{eq:TR_cauchy-step}
  \alpha^{\cc}_k=  
\left\{
    \begin{array}{ll}
        \frac{\Delta_k}{\chi_k^{1/2}\| g_k\|} & \text{ if } g_k^{\top}B_kg_k \le 0 ~ or~ \frac{g_k^{\top}B_kg_k}{\| g_k\|} \ge \frac{\Delta_k}{\chi_k^{1/2}} ,  \\
        \\
     \frac{g_k^{\top}B_kg_k}{\| g_k\|^2}  &  \text{ else}.
    \end{array}
\right.
\end{eqnarray}
where $\chi_k= \beta_k \left(\frac{5}{2}  - \frac{3}{2} \cos(\varpi_k)^2 + 2 \left(\frac{1- \cos(\varpi_k)^2}{\cos(\varpi_k)}\right)^2\right)$ and $\cos(\varpi_k) = \frac{g_k^{\top}s_k^Q}{\|g_k\| \| s_k^Q\|}$. $\Delta_k$ is initially equals to the current value of the TR radius (as in the original TR algorithm). For large values of $\alpha_k$ the sufficient decrease condition (\ref{sdcond}) may not be satisfied, in this case, the value of $\Delta_k$ is  contracted using the factor $\tau_1$. Iteratively, the value of $\alpha_k$ is updated and the acceptance condition (\ref{sdcond}) is checked again until its satisfaction. 
Algorithm~\ref{algo:LS-TR} details the adaptation of the classical TR algorithm when our proposed $M_k$-norm is used. We denote the final algorithm by LS-TR as it behaves as an LS algorithm.

 \LinesNumberedHidden
\begin{algorithm}[!ht]
\SetAlgoNlRelativeSize{0}
\caption{\bf LS-TR algorithm.}
\label{algo:LS-TR}
\SetAlgoLined
\KwData{select an initial point $x_0$ and the constants $0< \eta <1$, $0<\epsilon_d\le 1$, $0 \le \tau_1< 1 \le \tau_2$, and $0< \beta_{\min}< \beta_{\max}$.  Set the initial TR radius $\Delta_0>0$ and $\Delta_{\max}>\Delta_0$. }
\For{$k= 1,  2, \ldots$}{

Choose a parameter $\beta_k \in ]\beta_{\min}, \beta_{\max}[$\; 
Let $s^Q_k$ be an approximate stationary point of $m_k^Q$ satisfying  $|g_k^\top s^Q_k| \ge  \epsilon_d  \|g_k\|\|s^Q_k\|$\;
Set $\alpha_k$ and $\alpha_k^{\cc}$ using (\ref{eq:deltatr:k}) and (\ref{eq:TR_cauchy-step})\;
 \While{ {\normalfont condition (\ref{sdcond}) is not satisfied} }{
Set $\Delta_k \leftarrow  \tau_1 \Delta_k$, and update $\alpha_k$ and $\alpha_k^{\cc}$ using (\ref{eq:deltatr:k}) and (\ref{eq:TR_cauchy-step})\;
    }   
 Set $p_k= \alpha_k s^Q_k$, $x_{k+1}= x_k + p_k$ and  $ \Delta_{k+1} =\min\{\tau_2 \Delta_k,\Delta_{\max}\}$;
}
\end{algorithm}

As far as the objective function $f$ is continuously differentiable, its gradient is Lipschitz continuous, and its approximated Hessian $B_k$ is bounded for all iterations, the TR algorithm is globally convergent and will required a number of iterations of order $\epsilon^{-2}$ to produce a point $x_{\epsilon}$  with $\|\nabla f(x_{\epsilon})\| \le \epsilon$ \cite{SGratton_ASartenaer_PhLToint_2008}. 

We note that the  satisfaction of Assumption \ref{asm:1} is not problematic. As suggested for LS-ARC algorithm, one can modify the matrix $B_k$ using regularization techniques (as far as the Hessian approximation is kept uniformly bounded from above all over the iterations, the global convergence and the complexity bounds will still hold \cite{Conn_Gould_Toin_2000}).

\section{Numerical Experiments}
\label{section:5}

In this section we report the results of  experiments performed in order to assess the efficiency and
the robustness of the proposed algorithms (\texttt{LS-ARC} and \texttt{LS-TR}) compared with  the classical LS algorithm using the standard Armijo rule. 
In the latter approach, the trial step is of the form $s_k= \delta_k d_k $  where $d_k=s_k^Q$ if $-g_k^{\top}s^Q_k \ge \epsilon_d \|g_k\| \| s^Q_k\|$ ($s^Q_k$ is being an approximate stationary point of  $m_k^Q$) otherwise $d_k=-g_k$, and the step length $\delta_k>0$ is chosen such as
\begin{eqnarray} \label{cond:armijo}
f(x_k +s_k) \le f(x_k) + \eta s_k^{\top}g_k,
\end{eqnarray}
where $\eta \in ]0,1[$. The appropriate value of $\delta_k$ is estimated using a backtracking approach with a contraction factor set to $\tau \in]0,1[$ and where the step length is initially chosen to be $1$.  This LS method will be called \texttt{LS-ARMIJO}.
We implement  all the the algorithms as Matlab m-files and for all the tested algorithms  $B_k$ is set to the true Hessian $\nabla^2f(x_k)$, and $\epsilon_d=10^{-3}$ for both \texttt{LS-ARC} and \texttt{LS-ARMIJO}. 
Other numerical experiments (not reported here) with different values of $\epsilon_d$ (for instance $\epsilon_d = 10^{-1}$, $10^{-6}$, and $10^{-12}$) lead to almost the same results.

By way of comparison, we have also implemented the standard ARC/TR algorithms (see Algorithms \ref{algo:ARC} and \ref{algo:TR}) using the Lanczos-based solver \texttt{GLTR}/\texttt{GLRT} implemented in \texttt{GALAHAD} \cite{NIMGould_DOrban_PhLToint_2003}.The two subproblem solvers, \texttt{GLTR/GLRT} are implemented in Fortran and interfaced with Matlab using the default parameters. For the subproblem formulation we used the $\ell_2$-norm (i.e., for all iterations the matrix $M_k$ is set to identity). We shall refer to the ARC/TR methods based on \texttt{GLRT/GLTR}  as \texttt{GLRT-ARC/GLTR-TR}. 

The other parameters defining the implemented algorithms are set as follows, for  \texttt{GLRT-ARC} and \texttt{LS-ARC}
$$
\eta=0.1,  ~\nu_1=0.5, ~\nu_2=2,~\sigma_0=1, ~\mbox{and } \sigma_{\min}=10^{-16};   
$$
 for  \texttt{GLTR-TR} and \texttt{LS-TR}
$$
\eta=0.1, ~\tau_1= 0.5 , ~\tau_2=2, ~\Delta_0=1,~ \mbox{and } \Delta_{\max}=10^{16};   
$$
and last for \texttt{LS-ARMIJO}
$$
\eta=0.1, ~\mbox{and }  ~\tau= 0.5. \mbox{ } 
$$
 In all algorithms the maximum number of iterations is set to $10000$ and the algorithms stop when
$$
\|g_k\|\le 10^{-5}.
$$
 A crucial ingredient in \texttt{LS-ARC} and \texttt{LS-TR} is the management of the parameter $\beta_k$. A possible choice for $\beta_k$ is $|g_k^{\top} s_k^Q|/\|s^Q_k\|^2$. This choice is inspired from the fact that, when the Hessian matrix $B_k$ is SPD, this update corresponds to use the energy norm, meaning that the matrix $M_k$ is set equals to $B_k$
 (see \cite{Bergou_Diouane_Gratton_2017} for more details). However, this choice did not lead to good performance of the algorithms \texttt{LS-ARC} and \texttt{LS-TR}. In our implementation, for the \texttt{LS-ARC} algorithm, we set the value of $\beta_k$ as follows: $\beta_k=10^{-4} \sigma_k^{-2/3}$ if $g_k^{\top}s_k^Q < 0$ and $2$ otherwise. By this choice, we are willing to allow  \texttt{LS-ARC}, at the start of the backtracking procedure, to take approximately the Newton step. Similarly, for the \texttt{LS-TR} method, we set $\beta_k=1$ and this  allows \texttt{LS-TR}  to use the Newton step at the start of the backtracking strategy (as in \texttt{LS-ARMIJO} method). 

All the Algorithms are evaluated on a set of unconstrained optimization problems from the CUTEst collection \cite{Gould2015}. 
The test set contains $62$ large-scale ($1000 \le n \le 10000$) CUTest problems with their default parameters. 
Regarding the algorithms \texttt{LS-TR}, \texttt{LS-ARC}, and \texttt{LS-ARMIJO},  we approximate the solution of the linear system $B_ks=-g_k$ using the \texttt{MINRES} Matlab solver. The latter method is a Krylov subspace method designed to solve symmetric linear systems \cite{Paige_1975}. We run the algorithms with the  \texttt{MINRES} default  parameters except the relative tolerance error which is set  to  $10^{-4}$.  
We note that on the tested problems,  for \texttt{LS-ARC/LS-TR} , Assumption \ref{asm:1} was not violated frequently. The restoration of this assumption was ensured by performing iterations of \texttt{GLRT-ARC/GLTR-TR}  (with the $\ell_2$-norm) until a new successful iteration is found.

To compare the performance of the algorithms we use performance profiles proposed by Dolan and Mor\'e~\cite{EDDolan_JJMore_2002} over a variety of problems.  Given a set of problems $\mathcal{P}$ (of cardinality $|\mathcal{P}|$)
and a set of solvers $\mathcal{S}$, the performance profile
$\rho_s(\tau)$ of a solver~$s$ is defined as the fraction of problems
where the performance ratio $r_{p,s}$ is at most $\tau$
\begin{eqnarray*}
 \rho_s(\tau) \; = \; \frac{1}{|\mathcal{P}|} \mbox{size} \{ p \in \mathcal{P}: r_{p,s} \leq \tau \}.
\end{eqnarray*}
The performance ratio $r_{p,s}$ is in turn defined by
\[
r_{p,s} \; = \; \frac{t_{p,s} }{\min\{t_{p,s}: s \in \mathcal{S}\}},
\]
where $t_{p,s} > 0$ measures the performance of the solver~$s$ when solving problem~$p$
(seen here as the function evaluation, the gradient evaluation, and the CPU time).
Better performance of the solver~$s$,
relatively to the other solvers on the set of problems,
is indicated by higher values of $\rho_s(\tau)$.
In particular, efficiency is measured by $\rho_s(1)$ (the fraction of problems for which solver~$s$ performs
the best) and robustness is measured by $\rho_s(\tau)$ for $\tau$ sufficiently large
(the fraction of problems solved by~$s$). Following what is suggested in~\cite{EDDolan_JJMore_2002} for a better visualization,
we will plot the performance profiles in a $\log_2$-scale
(for which $\tau=1$ will correspond to $\tau=0$).

\begin{figure}[!ht]
\centering
\subfigure[Gradient evaluation.]{
\includegraphics[scale=0.44]{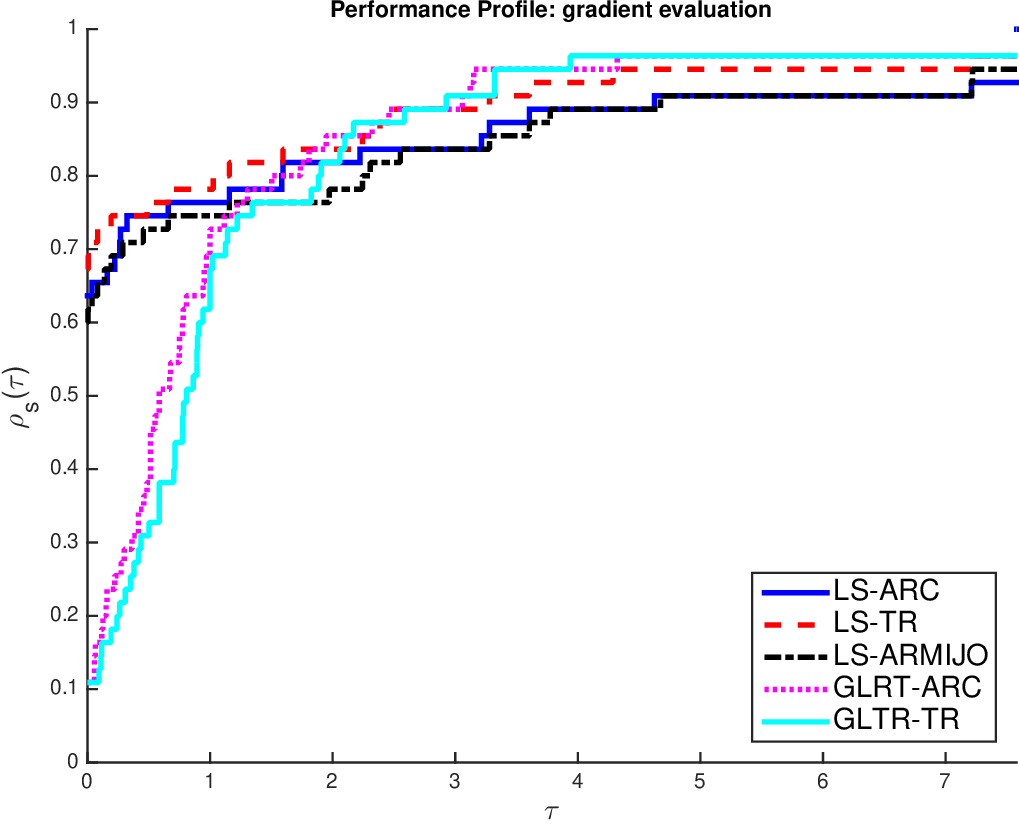}\label{subfig2:pp:ls}
}
\subfigure[Function evaluation.]{
\includegraphics[scale=0.44]{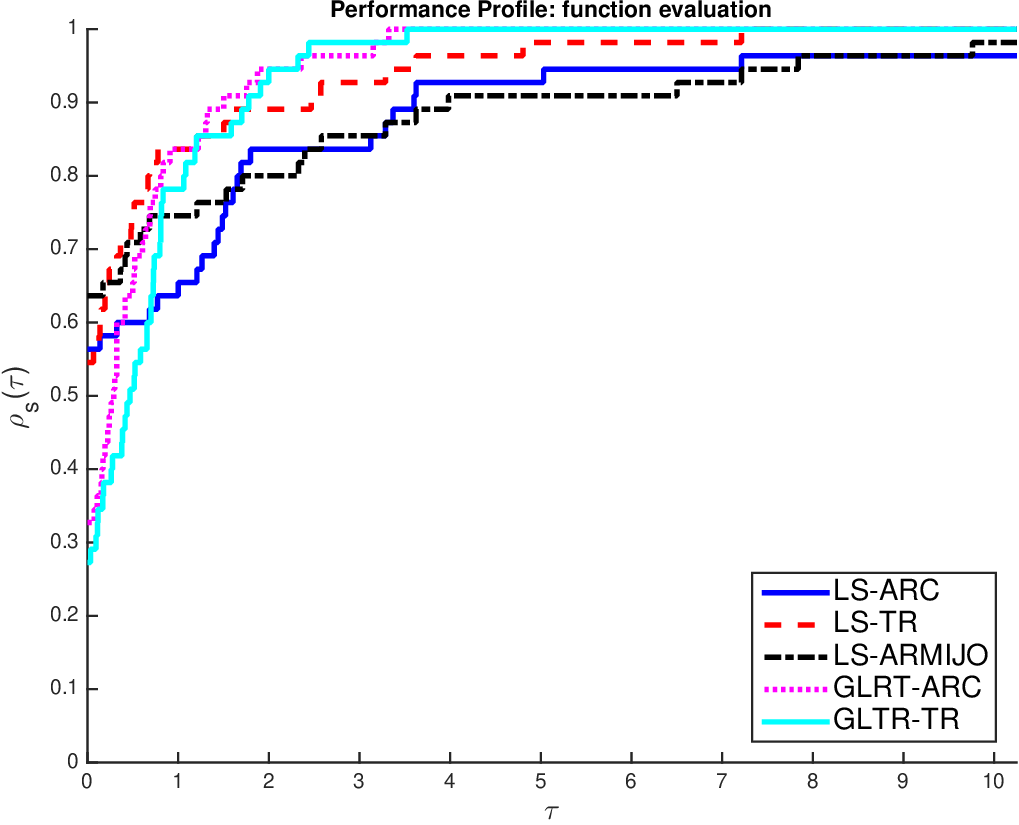} \label{subfig1:pp:ls}
}
\subfigure[CPU time.]{
\includegraphics[scale=0.44]{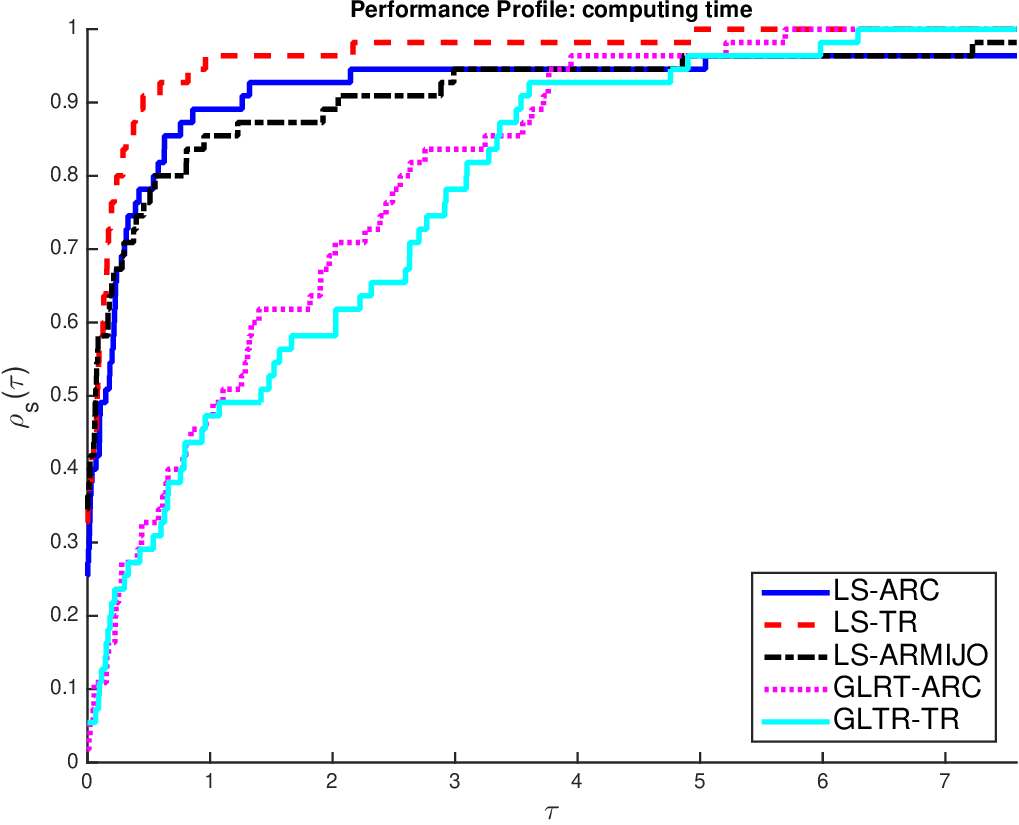} \label{subfig3:pp:ls}
}
\caption{Performance profiles for $62$ large scale optimization problems  (i.e., $1000 \le n \le 10000$).} \label{fig:pp:ls}
\end{figure}

We present the obtained performance profiles in Figure~\ref{fig:pp:ls}.
Regarding the gradient evaluation (i.e., outer iteration) performance profile, see Figure \ref{subfig2:pp:ls}, LS approaches are the most efficient among all the tested solvers (in more $60 \%$ of the tested problems LS methods  perform the best, while \texttt{GLRT-ARC}  and \texttt{GLTR-TR} are performing better only in less than $15 \%$). When it comes to robustness, all the tested approaches exhibit good performance, \texttt{GLRT-ARC} and \texttt{GLTR-TR} are  slightly better.

For function evaluation performance profile given by Figure~\ref{subfig1:pp:ls}, \texttt{GLRT-ARC} and \texttt{GLTR-TR} show a better efficiency but not as good as LS methods. In fact, in more than $50\%$ of the tested problems LS methods perform the best while \texttt{GLRT-ARC} and \texttt{GLTR-TR} are better only in less than $35\%$. 
The robustness of the tested algorithms is the same as in the gradient evaluation performance profile. 

In terms of the demanded computing time, see Figure~\ref{subfig3:pp:ls}, as one can expect, \texttt{GLRT-ARC} and \texttt{GLTR-TR} are turned to be very consuming compared to the LS approaches. In fact, unlike the LS methods where only an approximate solution of one linear system  is needed, the \texttt{GLRT}/\texttt{GLTR} approaches may require (approximately) solving multiple linear systems in sequence. 

For the LS approaches, one can see that \texttt{LS-TR} displays better performance compared to  \texttt{LS-ARMIJO} on the tested problems. The main difference between the two LS algorithms is the strategy of choosing the search direction whenever $g_k^{\top}s_k^Q>0$. In our tested problems, the obtained performance using \texttt{LS-TR} suggests that going exactly in the opposite direction $-s_k^Q$, whenever $s_k^Q$ is not a descent direction, can be seen as a good strategy compared to \texttt{LS-ARMIJO}.
\section{Conclusions}
\label{section:6}

In this paper, we have proposed the use of a specific norm in ARC/TR. 
With this norm choice, we have shown that the trial step of ARC/TR is getting collinear to the quasi-Newton direction. The obtained ARC/TR algorithm behaves as LS algorithms with a specific backtracking strategy.  Under mild assumptions, the proposed scaled norm was shown to be uniformly equivalent to the Euclidean norm. In this case, the obtained LS algorithms enjoy the same convergence and complexity properties as ARC/TR. We have also proposed a  second order version of the LS algorithm derived from ARC with an optimal worst-case complexity bound of order $\epsilon^{-3/2}$. 
Our numerical experiments showed encouraging performance of the proposed LS algorithms. 

A number of issues need further investigation, in particular the best choice and the impact of the parameter $\beta_k$ on the performance of the proposed LS approaches. Also, the analysis of the second order version of ARC  suggests that taking the Newton direction  is suitable for defining a line-search method with an optimal worst-case complexity bound of order $\epsilon^{-3/2}$. It would be interesting to confirm the potential of the proposed line search strategy compared to the classical LS approaches using extensive numerical tests.

   \small
\bibliographystyle{plain}
\bibliography{ref-tr-arc}
\end{document}